\documentclass{amsart}
\usepackage[utf8x]{inputenc}
\usepackage{amsmath}
\usepackage{amssymb}
\usepackage{psfrag}
\usepackage[all,cmtip]{xy}
\usepackage{amsmath,amscd}
\usepackage{amssymb}
\usepackage{graphicx}
\usepackage{wasysym}
\usepackage{textcomp}
\usepackage{enumerate}
\usepackage{graphicx} %If you want to include postscript graphics
\DeclareGraphicsExtensions{.pdf,.eps,.fig}
\newcommand{\Id}{\operatorname{Id}}

\newtheorem{definition}{Definition}[section]
\newtheorem{proposition}[definition]{Proposition}
\newtheorem{lemma}[definition]{Lemma}
\newtheorem{theorem}[definition]{Theorem}
\newtheorem{remark}[definition]{Remark}
\newtheorem{example}[definition]{Example}
\newtheorem{corollary}[definition]{Corollary}
\newtheorem{conjecture}[definition]{Conjecture}

\title{Relational symplectic groupoids
}

\author{Alberto S. Cattaneo and Ivan Contreras}

\address{Institut f\"ur Mathematik, Universit\"at Z\"urich Irchel, Winterthurerstrasse 190, CH-8057 Z\"urich, Switzerland}
\address{Department of Mathematics, University of California, Berkeley, California 94305, USA}
\thanks{A.S.C is partially supported by SNF Grant 20-149150. I.C. is supported by SNF Grant PBZHP2-147294}

\email{alberto.cattaneo@math.uzh.ch, ivan.contreras@berkeley.edu}
\begin{document}

\begin{abstract}
This note introduces the  construction of relational symplectic group\-oids as a way to integrate every Poisson manifold. 
Examples are provided and the equivalence, in the integrable case, with the usual notion of symplectic groupoid is discussed.
 
\end{abstract}
\maketitle

\tableofcontents
\section{Introduction}
Symplectic groupoids \cite{Weinstein} are Lie groupoids with a compatible symplectic structure. Their space of objects is naturally endowed with a Poisson structure. In a sense, a symplectic groupoid is a good symplectic replacement for the base Poisson manifold and is also related to its quantization.

This beautiful picture has one fault: namely, not every Poisson manifold arises as the space of objects of a symplectic groupoid. Moreover, the symplectic category is not the correct classical analogue of the category of Hilbert spaces which appears in quantum mechanics.

The goal of this paper is to define a more general structure, which we call a relational symplectic groupoid and of which ordinary symplectic groupoids are a particular case, as a ``groupoid object in the extended symplectic category." Quotation marks are needed as this as to be interpreted in the correct way. 

First, the ``extended symplectic category," whose objects are symplectic manifolds and whose morphisms are canonical relations (i.e., immersed Lagrangian submanifolds in the Cartesian product of symplectic manifolds with appropriate sign conventions for the symplectic form) is not a category since the composition of a canonical relation is not a submanifold in general. This is however not a problem since in the case at hand we are only interested in morphisms that compose well. (The situation is actually even subtler since, in order to have particularly interesting examples at hand, we include also infinite dimensional weak symplectic manifolds: the composition of canonical relations may then in general also fail to be Lagrangian.)

Second, a ``groupoid object" is roughly speaking obtained by replacing maps in the definition of a groupoid by canonical relations. Notice that at this level we only want to use diagrams involving the space of morphisms of the groupoid but not the space of objects, which already in the case of an ordinary symplectic groupoid is not a symplectic manifold but only Poisson. On the other hand, in order to have an interesting theory we have to introduce some extra axioms, which are automatically satisfied in the case of an ordinary symplectic groupoid (and which also have a natural interpretation in terms of a two-dimensional topological field theory).

Under some extra regularity conditions---we then speak of a regular relational groupoid---we are able to show that an appropriately defined ``space of objects" naturally carries a Poisson structure. Moreover, we show that every Poisson manifold arises as the
``space of objects" of a regular relational symplectic groupoid (even though for the classically nonintegrable Poisson manifolds we have to allow for infinite-dimensional relational symplectic groupoids). This integration of every Poisson manifold arises from the path space construction stemming from the Poisson sigma model \cite{Cat}.

Finally, there is a natural notion of morphisms (as structure compatible canonical relations) and equivalences (as morphisms whose transpose is also a morphism) between relational symplectic groupoids. We show that in the case of a classically integrable Poisson manifold the relational symplectic groupoid arising from the path space construction is canonically equivalent to any ordinary symplectic groupoid integrating it.

As a final remark notice \cite{Relational} that
the axioms for a relational symplectic groupoid make sense also in other categories, e.g., in the category of Hilbert spaces. This provides a definition of what the quantization of a relational symplectic groupoid should look like. We plan to return to this problem.
{}From this point of view, the relational symplectic groupoid approach is more natural than the stacky groupoid approach of \cite{TZ} even though, in
the nonintegrable case, one has to allow for infinite dimensional manifolds. Moreover, the flexibility we gain by the notion of equivalence might be useful for finding a better candidate for quantization than the ordinary symplectic groupoid as in \cite{Hawkins}.

\subsubsection*{Acknowledgement}
We thank F. Wagemann, A. Weinstein and M. Zambon for useful discussions and remarks.
A.S.C. thanks the University of California at Berkeley for hospitality.
\section{Relational symplectic groupoids}
Relational symplectic groupoids are objects in an extension of the usual symplectic category in which the objects are symplectic manifolds and the morphisms are symplectomorphisms.
This extension, which we will denote by $\mbox{\bf{Symp}}^{Ext}$, is not exactly a category, since the composition of morphisms is only partially defined; it corresponds to what in the 
literature is called a \emph{categoroid} \cite{Kashaev}.
In this section we will define such an extension and describe the 
relational symplectic groupoid in terms of an object and special morphisms in $\mbox{\bf{Symp}}^{Ext}$.
%\subsection{The extended symplectic category}
 
\subsection{The categoroid $\mbox{\bf{Symp}}^{Ext}$ }
In order to describe $\mbox{\bf{Symp}}^{Ext}$, we first need to include the case of infinite dimensional manifolds equipped with symplectic structures
\footnote{In this paper we restrict ourselves to the case of Banach manifolds. The construction in the Fr\'echet setting is treated in detail in \cite{thesis}.}.
\begin{definition}\label{Extended}\emph{
$\mbox{\textbf{Symp}}^{Ext}$ is a categoroid\footnote{ As we mentioned before, this is not an honest category since the composition of immersed canonical relations is not in general a smooth immersed submanifold. }  
in which the objects are weak symplectic manifolds, that is Banach manifolds equipped with a closed $2$-form $\omega$ such that the induced map 
\[\omega^{\sharp}: TM\to T^*M\]
is injective.
A morphism between two weak symplectic manifolds $(M,\omega_M)$ and $(N,\omega_N)$ is a pair $(L, \phi)$, where 
\begin{enumerate}[1.]
 \item $L$ is a smooth manifold.
\item $\phi\colon L \to M \times N$ is an immersion.
\footnote{Observe here that usually one considers embedded Lagrangian submanifolds, but we consider immersed ones.}
\item $T\phi_x$ applied to $T_xL$ is a Lagrangian subspace of $T_{\phi(x)}(\overline M\times N),\, \forall x \in L$.
\end{enumerate}
We will call these morphisms \textit{immersed canonical relations} and denote them by $L\colon M \nrightarrow N$. 
The partial composition of morphisms is given by composition of relations as sets.}
\end{definition}

\begin{remark}
\emph{
Observe that $\mbox{\textbf{Symp}}^{Ext}$ carries an involution $\dagger\colon(\mbox{\textbf{Symp}}^{Ext})^{op} \to  \mbox{\textbf{Symp}}^{Ext} $ that is the identity in objects and is the relational 
converse in morphisms, i.e. for $f\colon A \nrightarrow B$, $f^{\dagger}:= \{(b,a) \in B\times A \vert (a,b) \in f \}$.}
\end{remark}

\begin{remark}
\emph{This categoroid extends the usual symplectic category in the sense that the symplectomorphisms can be thought in terms of immersed canonical relations, namely, if $\phi\colon (M,\omega_M) \to (N, \omega_N))$ is a symplectomorphism between two weak symplectic manifolds, then $(\mbox{graph}(\phi), \iota)$, where $\iota$ is the inclusion of $\mbox{graph}(\phi)$ in $M \times N$, is a morphism in $\mbox{\textbf{Symp}}^{Ext}$.} 
\end{remark}

\subsection{Some special canonical relations}
We will describe in this subsection some particular canonical relations that will appear naturally in the construction of relational symplectic groupoid and in its connection to usual symplectic groupoids.\\

Following \cite{Alan}, consider a coisotropic subspace $W \subset V$. It follows that $W \oplus V$ is a coisotropic subspace of 
$\overline V \oplus V$ (since $(W \oplus V)^{\perp}= W^{\perp} \oplus \{0\}$). 
Since $\triangle_V \subset \overline V \oplus V$ is a Lagrangian subspace, where $\triangle _V$ denotes the diagonal of $V$ in $V\oplus V$, it follows that $P_{W \oplus V}(\triangle_V)$ is a Lagrangian subspace of
 $\underline{\overline W \oplus V}= \overline{\underline{W}}\oplus V$, where $P_{W \oplus V}$ denotes the quotient map $P_{W \oplus V}: \overline{V}\oplus V \to \underline{W \oplus V}$. This projection will be denoted by $I$ and is a canonically defined canonical relation $I\colon \underline W \nrightarrow V$. 
In fact, this also holds in the infinite dimensional setting due to the following
\begin{proposition} \label{projection}
\emph{For any (possibly infinite dimensional) symplectic space $V$, $I$ is a canonical relation.}
\end{proposition}
\begin{proof}
Explicitely we have that
\[I=\{([w],w) \in (W \diagup W^{\perp})\times V \vert w \in W \},\]
therefore
\[I^{\perp}=\{([v],z)\vert [v] \in W \diagup W^{\perp}, z\in V, \underline{\omega}([v],[w])-\omega(z,w)=0, \forall w \in W\}.\] 
By linearity  and the construction of the reduction this is equivalent to
\begin{eqnarray*}
I^{\perp}&=&\{([v],z)\vert [v] \in W \diagup W^{\perp}, z\in V, \omega(v-z,w)=0, \forall w \in W, v \in [v]\}\\
&=& \{([v],z)\vert [v] \in W \diagup W^{\perp}, \, z \in V, v-z \in W^{\perp}, \forall v \in [v]\}.
\end{eqnarray*}
This implies in particular that $z$ and $v$ belong to the same equivalence class and since $v \in W$ and $W$ is 
coisotropic it follows that $z \in W$ and therefore $[v]=[z]$, hence $I^{\perp}=I$, as we wanted.
\end{proof}

 We denote by $P:= I^{\dagger}: 
V \nrightarrow \underline{W}$, the transpose of $I$. We can then prove  the following
\begin{proposition}
 \emph{The following relations hold 
\begin{enumerate}[1.]
 \item $P\circ I\colon \underline{W}\nrightarrow \underline W = \mbox{Graph(Id)}.$
 \item $I\circ P\colon V \nrightarrow V= \{(w,w^{'})\in V \times V \vert  w, w^{'} \in W;\, [ w]= [ w^{'}]\}.$ Furthermore, $I\circ P \subset W \times W$ is
 an equivalence relation on $W$ and \[W \diagup I\circ P= \underline W.\]
\end{enumerate}
}
\end{proposition}
\begin{proof}
Direct computation. 
\end{proof}
The following lemma (from Proposition A.1 in Appendix A of \cite{AlbertoPavel2}) will be important for the rest of the paper and it relates Lagrangian subspaces before and after symplectic reduction.
\begin{lemma}\label{Coisotropic}
\emph{Let $(V,\omega)$ be a symplectic space. Let $C$ be a coisotropic subspace of V. Let $L$ be a subspace such that 
\[C^{\perp} \subset L\subset C \subset V. \]
Assume that $\underline{L}:= L / C^{\perp}$ is Lagrangian in $\underline{C}:= C / C^{\perp}.$ Then, $L$ is a Lagrangian subspace of $V$.}
\end{lemma}

\begin{proof}
 Let $p\colon C \to \underline{C}$ denote the projection map to the reduced space.
The idea is to prove that $L=L^{\perp}$ using the fact that $p^{-1}(\underline{L})=p^{-1}(\underline{L}^{\perp})$. First, we prove that 
$p^{-1}(\underline{L}^{\perp})=L^{\perp}$. Let $v \in p^{-1}(\underline{L}^{\perp})$. We have that, by definition,
\[\underline{\omega}(p(v),\underline{w})=0, \forall \underline{w} \in \underline{L} \]
and this implies that
\[ \omega(v,w)=0, \forall w\in V \vert \underline{w} \in \underline{L},\]
therefore,
\[\omega(v,w)=0, \forall w \in L.\]

Also by definition we have that 
\[p^{-1}(\underline{L})=L,\]
therefore, using the fact that $\underline{L}$ is Lagrangian, $\underline{L}^{\perp}=\underline{L}$, 
we obtain that $L=L^{\perp},$ as we wanted.
\end{proof}

As an application of Proposition \ref{projection}  and Lemma \ref{Coisotropic} we have the following 

\begin{definition}
 \emph{Let $\underline L\colon \underline W \nrightarrow \underline W$ be a canonical relation. By Proposition \ref{Coisotropic},
 \[l(\underline L):= I \circ \underline L \circ P\colon V \nrightarrow V\] 
is an isotropic relation.  Moreover, since $(0,0) \in L$, we check that 
\[ W^{\perp} \oplus W^{\perp} \subset l(\underline L) \subset C\oplus C \subset \overline{V}\oplus V,\]
in addition if $p: \overline V \oplus V \to \overline W \oplus W$ denotes the symplectic reduction, then $p(l(\underline L)= L)$
therefore, we can apply Lemma \ref{Coisotropic} and we can conclude that $l(\underline L)$ is Lagrangian. The canonical relation $l(\underline L) $is also called the \emph{canonical lift} of $\underline L$. 
}
\end{definition}
\begin{definition}\label{canonicalproj}
\emph{Let $L\colon V \nrightarrow V$ be a canonical relation. Then 
\[p(L):= P \circ L \circ I\colon \underline W \nrightarrow \underline W\]
is Lagrangian, due to the fact that the symplectic projection of Lagrangian subspaces is Lagrangian (for the proof see e.g. Lemma 5.2 in \cite{Alan}). This canonical relation is also called the \emph{canonical projection} of $L$}.
\end{definition}
These two particular relations have interesting properties. Observe first that the composition $p\circ l$ is the identity for Lagrangian subspaces of 
the symplectic reduction $\underline W$, however, the composition $l \circ p$ is not the identity.

\subsection{The main construction}
This section contains the general description of relational symplectic groupoids, defined as special objects in $\mbox{\textbf{Symp}}^{Ext}$. It is a way to model the space of boundary fields before reduction of the PSM and to define a more general version of integration of Poisson manifolds. We give the main definitions, we discuss the connection with Poisson structures and we give some natural examples. For the motivational example of $T^*PM$ we prove that in fact we obtain relational symplectic groupoids for any Poisson manifold $M$ and we explain geometrically the integrability conditions for $T^*M$ in terms of the immersed canonical relations defining the relational symplectic groupoid.
\begin{definition}\emph{
A \textbf{relational symplectic groupoid}  (short RSG) is a triple $(\mathcal G,\, L,\, I)$ where 
\begin{enumerate}[1.]
 \item $\mathcal G$ is a weak symplectic manifold with a weak symplectic form $\omega$.
\item $L$ is an immersed Lagrangian submanifold of $\mathcal G ^3.$
 \item $I$ is an antisymplectomorphism of $\mathcal G$ called the \emph{inversion},
\end{enumerate}
satisfying the following six axioms A.1-A.6\footnote{In the case that $\mathcal G$ is finite dimensional, the Lagrangian conditions in axioms A.4, A.5 and A.6 are automatically satisfied.}: 
(A graphic interpretation of these axioms is given at the end of the section)}
\end{definition}

\begin{itemize}
 \item \textbf{\underline{A.1}} $L$ is cyclically symmetric, i.e. if $(x,y,z) \in L$, then $(y,z,x) \in L$.

\item \textbf{\underline{A.2}} $I$ is an involution (i.e. $I^2=\Id$).\\

\textbf{\underline{Notation}} $L$ is an immersed canonical relation $\mathcal G \times \mathcal G \nrightarrow \bar{\mathcal G}$ and will 
be denoted by $L_{rel}.$ Since the graph of $I$ is a Lagrangian submanifold of $\mathcal G \times \mathcal G$, $I$ is an immersed canonical 
relation $\bar {\mathcal G} \nrightarrow \mathcal G$ and will be denoted by $I_{rel}$.\\
$L$ and $I$ can be regarded as well as immersed canonical relations 
\[ \bar {\mathcal G} \times \bar{\mathcal G} \nrightarrow \mathcal G \mbox{ and } \mathcal G \nrightarrow \bar{\mathcal G}\]
respectively, which will be denoted by $\overline{L_{rel}}$ and $\overline{I_{rel}}.$ The transposition 
\begin{eqnarray*} 
T\colon  \mathcal G \times \mathcal G &\to& \mathcal G \times \mathcal G\\
(x,y) &\mapsto& (y,x)
\end{eqnarray*}
induces canonical relations
\[T_{rel}\colon \mathcal G \times \mathcal G \nrightarrow \mathcal G \times \mathcal G \mbox{ and } \overline{T_{rel}}: 
\bar {\mathcal G}\times \bar{\mathcal G} \nrightarrow \bar {\mathcal G}\times \bar{\mathcal G}.\]
The identity map $Id\colon  \mathcal G \to \mathcal G$ as a relation will be denoted by $\Id_{rel}\colon  \mathcal G \nrightarrow \mathcal G$ and by $\overline{\mbox{Id}_{rel}}\colon  \overline{\mathcal G} \nrightarrow \overline{\mathcal G}. $

Since $I$ and $T$ are diffeomorphisms, it follows that $I_{rel} \circ L_{rel}$ and $\overline{L}_{rel} \circ \overline{T}_{rel}\circ (\overline{I_{rel}} \times \overline{I_{rel}})
$ are immersed submanifolds.  For a relational symplectic groupoid we want that these two compositions to be morphisms $\mathcal G \times \mathcal G \nrightarrow \mathcal G$, and moreover we want them to coincide.

\item \textbf{\underline{A.3}} 
The compositions $I_{rel} \circ L_{rel}$ and $\overline{L}_{rel} \circ \overline{T}_{rel}\circ (\overline{I_{rel}} \times \overline{I_{rel}})$ are immersed Lagrangian submanifolds of $\mathcal G^3$ 
and
\begin{equation*}
I_{rel} \circ L_{rel}= \overline{L}_{rel} \circ \overline{T}_{rel}\circ (\overline{I_{rel}} \times \overline{I_{rel}}).
\end{equation*}
%\\

%, both sides of the equaliy correspond to immersed Lagrangian submanifolds.\\

Now, define \[L_3:= I_{rel} \circ L_{rel}\colon  \mathcal G \times \mathcal G \nrightarrow \mathcal G.\] As a corollary of the previous axioms we get that
\begin{corollary}
$\overline{I_{rel}}\circ L_3 = \overline{L_3} \circ \overline{T_{rel}}\circ(\overline{I_{rel}}\times \overline{I_{rel}}).$ 
\end{corollary}
\begin{proof}By Axiom A.2 and by definition of $L_3$, the left hand side of the equation, as a relation from $\mathcal G \times \mathcal G$ to $\overline {\mathcal G}$ 
can be rewritten as 
\begin{equation}\label{A}
\overline{I_{rel}}\circ L_3= \overline{Id_{rel}}\circ L_{rel}.
\end{equation}
In the right hand side, by axiom A.3, we can rewrite
\begin{eqnarray}
\overline{L_3} \circ \overline{T_{rel}}\circ(\overline{I_{rel}}\times \overline{I_{rel}})&=&
\overline{I_{rel}}\circ (\overline{L_{rel}}\circ \overline{T_{rel}}\circ (\overline{I_{rel}}\times \overline{I_{rel}}))\\
&=&\overline{I_{rel}}\circ I_{rel} \circ L_{rel}\\
&=& \overline{Id_{rel}} \circ L_{rel}.\label{B} 
\end{eqnarray}
Comparing (\ref{A}) and (\ref{B}) we obtain the desired result.
\end{proof}

\item \textbf{\underline{A.4}} 
\begin{enumerate}[1.]
\item The compositions $L_3 \circ (L_3 \times \Id)$ and $L_3\circ (\Id \times L_3)$ are immersed submanifolds of $\mathcal G^4$.
\item  $L_3 \circ (L_3 \times \Id)$ and $L_3\circ (\Id \times L_3)$ are Lagrangian submanifolds of $\overline{\mathcal G^3}\times \mathcal G$.
\item The following equality holds \begin{equation}\label{asso} 
L_3 \circ (L_3 \times \Id)= L_3\circ (\Id \times L_3)\end{equation}
\end{enumerate}
%is a morphism $\colon  \mathcal G ^3 \nrightarrow  \mathcal G$. \\

%\begin{remark}\emph{The part 2 of A.4. follows automatically in the finite dimensional case from the fact that, since $I$ is an antisymplectomorphism, its graph is Lagrangian, 
%therefore $L_3$ is Lagrangian, and so $(\Id \times L_3)$ and $(L_3 \times \Id).$}
%\end{remark}
The graph of the map $I$, as a relation $* \nrightarrow \mathcal G \times \mathcal G$ will be denoted by $L_I$.\\
\item \textbf{\underline{A.5}} 
\begin{enumerate}[1.]
\item The compositions
$L_3 \circ L_I$ and $L_3\circ(L_3\circ L_I\times L_3\circ L_I)$ are immersed submanifolds of $\mathcal G$. %is a morphism $* \nrightarrow \mathcal G$ and if we define
% \begin{equation}L_1:= L_3 \circ L_I\colon  * \nrightarrow \mathcal G,
% \end{equation} 
\item $L_3 \circ L_I$ and $L_3\circ(L_3\circ L_I\times L_3\circ L_I)$ are Lagrangian submanifolds of $\mathcal G$.
\item Denoting by $L_1$ the morphism $L_1:= L_3\circ L_I \colon * \nrightarrow \mathcal G$, then
\begin{equation} \label{unit}
 L_3\circ(L_1 \times L_1)= L_1.
\end{equation}
 \end{enumerate}

 From the definitions above we get the following
\begin{corollary}\emph{
\[\overline{I_{rel}}\circ L_1= \overline{L_1},\]
that is also equivalent to 
\[I(L_1)= \overline{L_1},\]
where $L_1$ is regarded as an immersed Lagrangian submanifold of $\mathcal G$.}
\end{corollary}

\begin{proof}
We have that
\begin{eqnarray*}
\overline{I_{rel}}\circ L_1&=& \overline{I_{rel}}\circ L_3 \circ L_1\\
 &\stackrel{\mbox{Cor.1}}{=}& \overline{L_3} \times \overline{T_{rel}} \circ (\overline{I} \times \overline{I})\circ L_I.
\end{eqnarray*}
As sets, we have that
\begin{eqnarray*}
\overline{T_{rel}}\circ (\overline{I} \times \overline{I})\circ L_I&=& T(I\times I (L_I))\\
L_I&=&\{(x,\, I(x)), \, x \in \mathcal G\}\\
I\times I (L_I)&=& \{(I(x), I^2(x)),\, x\in \mathcal G \}\\
&\stackrel{A.2}{=}&\{(I(x), x),\, x\in \mathcal G\}\\
T(I\times I (L_I))&=& \{(x, I(x)),\, x \in \mathcal G\}= L_I. 
\end{eqnarray*}
From the last equation we get
\[\overline{T_{rel}}\circ (\overline{I} \times \overline{I})\circ L_I= \overline{L_1}
\] and therefore
\[\overline{I_{rel}}\circ L_1 = \overline{L_3} \circ \overline{L_I}= \overline{L_1}.\]

\end{proof}

\item \textbf{\underline{A.6}} 

\begin{enumerate}[1.]
\item $L_3\circ (L_1 \times \Id)$ and $L_3\circ (\Id \times L_1)$ are immersed submanifolds of $\mathcal G \times \mathcal G.$
\item $L_3\circ (L_1 \times \Id)$ and $L_3\circ (\Id \times L_1)$ are Lagrangian submanifolds of $\overline{\mathcal G}\times \mathcal{G}$.
\item If we define the morphism $$L_2:=L_3\circ (L_1 \times \Id)\colon  \mathcal G \nrightarrow \mathcal G,$$
then the following equations hold
\begin{enumerate}[1.]
\item
\begin{equation*}
L_2=L_3\circ (\Id \times L_1).
\end{equation*}
\item $L_2$ leaves $L_1$ and $L_3$ invariant, i.e.
\begin{eqnarray} 
L_2\circ L_1&=& L_1\label{invariance1}\\
L_2\circ L_3&=& L_3\circ (L_2 \times L_2)=L_3\label{invariance3}.
\end{eqnarray}
\item 
\begin{equation}
\overline{I_{rel}}\circ L_2= \overline{L_2}\circ \overline{I_{rel}}\mbox{ and } L_2^{\dagger}=L_2.\label{symmetric}
\end{equation}
\end{enumerate}
\end{enumerate}
\begin{corollary}\label{invariance2}\emph{$L_2$ is idempotent, i.e.}
\begin{equation}\label{invari2}
L_2\circ L_2= L_2. 
\end{equation}
\end{corollary}
\begin{proof}
It follows directly from the definition of $L_2$ and Equations \ref{asso} and \ref{unit}.
\end{proof}

\begin{remark}
\emph{The following is an interpretation of the axioms of the relational symplectic groupoid (for simplicity we present them in the case when $\mathcal G$ is a group):}
\emph{
\begin{itemize}
\item The cyclicity axiom  (A.1) encodes the cyclic behavior of the multiplication and inversion maps for groups, namely, if $a,b,c$ are elements of a group $G$ with unit $e$ such that $abc=e$, then $ab=c^{-1}, \, bc= a^{-1}, ca=b^ {-1}$.%The horizontal edges of the boundary of each surface represent the relational symplectic groupoid $\mathcal G$  and the interior of each surface represents the canonical relation.
\item (A.2) encodes the involutivity property of the inversion map of a group, i.e. $(g^{-1})^{-1}=g, \forall g \in G$.
\item (A.3) encodes the compatibility between multiplication and inversion:
$$(ab)^{-1}=b^{-1}a^{-1}, \forall a,b \in G.$$
\item (A.4) encodes the associativity of the product: $a(bc)=(ab)c, \forall a,b,c \in G$.
\item (A.5) encodes the property of the unit of a group of being idempotent: $ee=e$.
\item The axiom (A.6)  states an important difference between the construction of relational symplectic groupoids and usual groupoids. The compatibility between the multiplication and the unit is defined up to an equivalence relation, denoted by $L_2$, whereas for groupoids such compatibility is strict; more precisely, for groupoids such equivalence relation is the identity. 
In addition, the multiplication and the unit are equivalent with respect to $L_2$.
\end{itemize}
}
\emph{This description explains why the choice of the axioms of the relational symplectic groupoid are \emph{natural}
\footnote{Notice that any discrete group is actually a 
relational symplectic groupoid and also a symplectic groupoid with the zero symplectic structure.}.}
\end{remark}

\begin{remark} \label{counter1}
\emph{Equations \eqref{unit}, \eqref{invariance1}, \eqref{invariance3}, \eqref{symmetric} and \eqref{invari2} have to be stated as part of the axioms and they cannot be deduced as corollaries. Here there is an example of a structure that satisfies  the axioms from A.1. to A.4 but not A.5 or A.6.
\begin{enumerate}[1.]
\item $\mathcal G = \mathbb Z$ (as a non connected zero dimensional symplectic manifold)
\item $L=\{(n,m,-n-m-1) \in \mathbb Z ^3\}$
\item $I\colon  n \mapsto -n$
\end{enumerate}
For this example, the spaces $L_i$ are given by
\begin{eqnarray*}
L_1&=& \{ 1\}\\
L_2&=& \{ (m,m+2)\mid m\in \mathbb Z\}\\
L_3&=&\{(m,n,m+n+1)\mid n,m \in \mathbb Z ^2 \}
\end{eqnarray*}
for which we get that
\begin{eqnarray*}
L_3\circ(L_1 \times L_1)&=&\{3\} \neq L_1\\
L_2\circ L_1&=&\{3\} \neq L_1\\
L_2\circ L_2&=&\{(m,m+4)\mid m\in \mathbb Z\} \neq L_2\\
L_2\circ L_3&=&\{(m,n,m+m+3)\mid m,n \in \mathbb Z\} \neq L_3\\
\overline{I_{rel}}\circ L_2&=&(m, -m-2) \neq (m,-m+2)= \overline{L_2} \circ \overline {I_{rel}}.
\end{eqnarray*}
\begin{remark}\emph{
This counterexample has also a finite set version, replacing $\mathbb Z$ by $\mathbb Z / k \mathbb Z$, with $k \geq 3$.}
\end{remark}
}

\end{remark}

%\textit{Proof: }
 
%Denoting by $L_2^{rel}$ the submanifold $L_3\circ (Id \times L_1)$ as a relation $* \nrightarrow \mathcal G \times \mathcal G$ 
\end{itemize}

\begin{figure}
%\psfrag{Rq}{$\mathbb{R}^q$}
\centering%
\center{\includegraphics[scale=0.50]{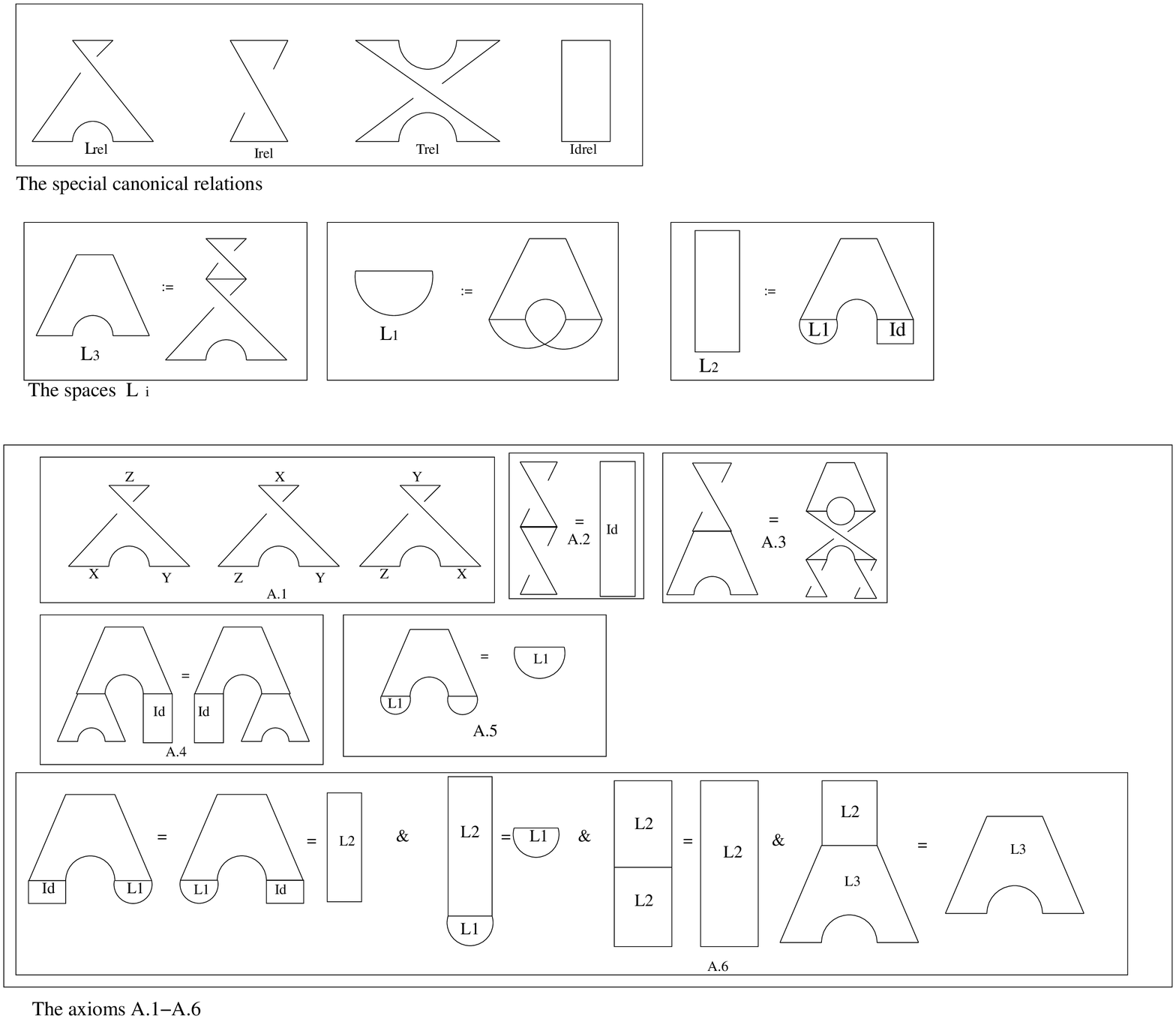}}
\caption{\underline{Relational symplectic groupoid: Diagrammatics.} The morphisms 
$I_{rel}$ and $L_{rel}$ are represented by a twisted stripe and pair of \emph{paper doll pants} respectively, 
and the induced immersed canonical relations $L_1, L_2$ and $L_3$ are constructed as compositions of $L$ and $I$. 
As it is shown in the Figure, they should satisfy  the previously defined compatibility axioms. The horizontal segments in the boundary 
of the surfaces represent the weak symplectic manifold $\mathcal G$. the non horizontal segments have no meaning.}
\label{fig:Ax}
\end{figure}
\subsection{The categoroid \textbf{RSGpd}}
We have stated so far the notion of relational symplectic groupoids in the extended symplectic category. These objects have a natural notion of morphism that is also defined in the context of canonical relations. Hence, as before, the composition of morphisms is only partially defined but it allows us to describe the categoroid  \textbf{RSGpd} of relational symplectic groupoids with suitable morphisms.
\begin{definition} \emph{
A morphism between two relational symplectic  groupoids $(\mathcal G,L_{\mathcal G},I_{\mathcal G})$ 
and $(\mathcal H,L_{\mathcal H},I_{\mathcal H})$ 
is a relation $F\colon  \mathcal G\nrightarrow \mathcal H$ satisfying the following properties:
\begin{enumerate}[1.]
 \item $F$ is an immersed Lagrangian submanifold of $\mathcal G\times \bar{\mathcal H}$.
 \item $F\circ I_{\mathcal G}=I_{\mathcal H} \circ F$. 
 \item $L_{\mathcal H}\circ (F\times F)=F\circ L_{\mathcal G}.$
\end{enumerate}
}
\end{definition}

\begin{definition}\emph{
A morphism of relational symplectic groupoids $F\colon  \mathcal G \to \mathcal H $ is called an \textbf{equivalence} 
if the transpose canonical relation $F^{\dagger}$ is also a morphism.}
\end{definition}
\begin{remark}\label{equiv}\emph{From the definition, it follows that an equivalence $F$ satisfies the following compatibility conditions with respect to $L_1$ and $L_2$:
\begin{eqnarray*}
F\circ (L_1)_{\mathcal G}&=& (L_1)_{\mathcal H}\\
F\circ (L_1)_{\mathcal H}&=& (L_1)_{\mathcal G}\\
F^{\dagger}\circ F&=& (L_2)_{\mathcal G} \label{1}\\
F\circ F^{\dagger}&=& (L_2)_{\mathcal H}. \label{2}.
\end{eqnarray*}
} \end{remark}

The following are some examples of equivalences.
\begin{example}\emph{Let $(\mathcal G, L, I)$ be a relational symplectic groupoid. Then $L_2$ is an equivalence between $\mathcal G$ and itself. \\
To check that $L_2$ is a morphism of relational symplectic groupoids, we observe that, by Equation \ref{symmetric} $L_2$ commutes with $I$ and by Equation \ref{invariance3} we get that 
\begin{eqnarray*}
L_{\mathcal G}\circ (L_2 \times L_2)&=& I\circ L_3\circ (L_2\times L_2)\\
                                                                &=& I \circ L_3= I \circ L_2 \circ L_3\\
                                                                &=& L_2 \circ I \circ L_3= L_2 \circ L_{\mathcal G}. 
                                                                \end{eqnarray*} 
Since $L_2$ is self transposed, it follows that $L_2$ is an equivalence.                                                                
}
\end{example}
\begin{example}\emph{
For a relational symplectic groupoid $(\mathcal G, L, I)$ the map $I$ is an equivalence from $(\mathcal G, L,I)$ to $(\overline {\mathcal G}, L^{'}, I)$, where $L^{'}= L \circ T_{rel}$}.
\end{example}
In the next section we give additional examples of equivalences, after describing some special case of relational symplectic groupoids.

\section{The regular case}
The next set of axioms defines a particular type of relational symplectic groupoids  which will allow us to relate the construction of relational symplectic groupoids for Poisson manifolds to the usual symplectic groupoids for the integrable case. Before this, we introduce the notion of immersed coisotropic submanifolds for weak symplectic manifolds.
\begin{definition}\emph{ An immersed coisotropic submanifold of a weak symplectic manifold $M$ is a pair $(\phi, C)$ such that
\begin{enumerate}[1.]
 \item $C$ is a smooth Banach manifold.
 \item $\phi\colon C \to M$ is an immersion.
%\footnote{Observe here that usually one considers embedded Lagrangian submanifolds, but we consider immersed ones.}
\item $T\phi_x$ applied to $T_xC$ is a coisotropic subspace of $T_{\phi(x)}M,\, \forall x \in C$.
\end{enumerate}}
\end{definition}
\begin{definition}\emph{
A relational symplectic groupoid $(\mathcal G,\, L,\, I)$ is called \textbf{regular} (short RRSG) if the following three axioms A.7, A.8 and A.9 are satisfied. Consider $\mathcal G$ 
as a relation $* \nrightarrow \mathcal G$ denoted by $\mathcal G_{rel}$.}
\end{definition}
\begin{itemize}
 \item \textbf{\underline{A.7}} 
 \begin{equation} \label{C}
 C:=L_2 \circ \mathcal G_{rel}
 \end{equation} 
 is an immersed submanifold of $\mathcal G$. 
\end{itemize}
\begin{corollary} \label{equi}\emph{
$L_2$ is an equivalence relation in $C$}.
\end{corollary}
\begin{proof}
By Equation \ref{invariance2}
\begin{equation*}
L_2=L_2\circ L_2 \subset L_2\circ (\mathcal G\times \mathcal G)= C\times C\colon  *\nrightarrow \mathcal G \times \mathcal G,
\end{equation*}
so $L_2$ is a relation on $C$. By Equation \ref{invariance2}, $L_2$ is transitive, by Equation \ref{symmetric} it is symmetric and, for any $x \in C$, by definition, there exists $y$ such that $(x,y)\in L_2$ and by symmetry and transitivity of $L_2$, we conclude that $(x,x)\in L_2$, hence, $L_2$ is an equivalence relation.
\end{proof}
\begin{corollary}\emph{
$C$ is an immersed coisotropic submanifold of $\mathcal G$.}
\end{corollary}
\begin{proof}
By definition of $C$ we get that 
\[TL_2 \subset TC\oplus TC\] 
and by A.6. $TL_2$ is a Lagrangian subspace of $\overline{T\mathcal G} \oplus T\mathcal G$.  Therefore 
\[TC^{\perp}\oplus TC^{\perp}\subset TL_2\subset TC\oplus TC.\]
If we restrict to the diagonal $\triangle \mathcal G$, we get that
\[TC^{\perp}\cong (TC^{\perp}\oplus TC^{\perp})\cap \triangle \mathcal G \subset (TC\oplus TC)\cap \triangle \mathcal G \cong TC,\] hence 
$C$ is coisotropic.
%By Equation \ref{invariance1} we conclude that $L_1\subset C$, hence $C$ is coisotropic.
\end{proof}
The following Proposition allows us (in principle at the infinitesimal level), to regard the equivalence relation given by $L_2$ as the equivalence relation given by the characteristic 
foliation of $C$.
\begin{proposition}\emph{
Let 
$$R^C:=\{(x,y)\in C\times C \mid L_x=L_y\},$$
where $L_x$ is the leaf of the characteristic foliation through the point $x \in C$. Let $(x,y)\in R^C\cap L_2$.
Then
\begin{equation}\label{char}
T_{(x,x)}R^C=T_{(x,x)}L_2.
\end{equation}}
\end{proposition}
\begin{proof}
First we will prove that $T_{(x,x)}R^C\subset T_{(x,x)}L_2$. For this, consider $(X,Y)\in T_{(x,x)}R^C$, since $X-Y \in T_xC^{\perp}$, we get that 
\begin{equation}\label{E1}
(X-Y,0) \in T_{(x,x)}C^{\perp}\oplus TC^{\perp}.
\end{equation}
Since $L_2\subset C\times C$ and $L_2$ is Lagrangian
\begin{equation}\label{E2}
T_{(x,x)}C^{\perp}\oplus T_{(x,x)}C^{\perp} \subset T_{(x,x)}L_2\subset T_xC \oplus T_xC. 
\end{equation} 
Combining Equations \ref{E1} and \ref{E2}, we get that
\begin{equation}\label{E3}
(X-Y,0)\in T_{(x,x)}L_2.
\end{equation}
Since the diagonal $\triangle C$ is contained in $L_2$ (from Corollary \ref{equi}), then 
\begin{equation}\label{E4}
(Y,Y)\in T_{(x,x)}L_2, \forall Y \in C.
\end{equation}
From equations \eqref{E1} and \eqref{E4}, we conclude that 
\begin{equation*}
(X,Y)=(X-Y,0)+(Y,Y)\in T_{(x,x)}L_2,
\end{equation*}
as we wanted.
Now we prove that $T_{(x,x)}R^C$ is a Lagrangian subspace of $T_xC\oplus T_xC$. For this, first observe that 
\begin{equation*}
T_xC^{\perp} \oplus T_x^{\perp}C \subset T_{(x,x)}R^C \subset T_xC \oplus T_xC
\end{equation*}
and that the canonical projection of $T_{(x,x)}R^C$ in the symplectic reduction $\underline C \oplus \underline C$ is $\triangle C$, 
which is Lagrangian.\\
Applying lemma \ref{Coisotropic} we conclude that $T_{(x,x)}R^C$ is Lagrangian. Now, since $T_{(x,x)}L_2$ is also Lagrangian by the axioms above and it contains $T_{(x,x)}R^C$ as a subspace, it follows that 
\begin{equation*}
T_{(x,x)}L_2=T_{(x,x)}L_2^{\perp}\subset T_{(x,x)}R^C=T_{(x,x)}R^C,
\end{equation*}
hence $T_{(x,x)}L_2=T_{(x,x)}R^C$, as we wanted.
\end{proof}

%\begin{itemize}
%\item \textbf{\underline{A.8}} Let  
%$$L_2^c:=\{(x,y)\in C\times C \mid L_x=L_y\},$$
%where $L_x$ is the leaf of the characteristic foliation through the point $x \in C$. Then
%\begin{equation}\label{char}
%L_2^c=L_2.
%\end{equation}
%\end{itemize}
%\begin{corollary} 

%Setting $C:= L_2 \circ \mathcal G_{rel}$ the following corollary holds.\\
%\begin{enumerate}[1.]
%\item \[C^*= \mathcal G ^* \circ L_2\] 
% \item $L_2$ defines an equivalence relation on $C$.
 %\item This equivalence relation is the same as the one given by the characteristic foliation on $C$. 
%\end{enumerate}
%\end{corollary}
%\textit{Proof: }
%\begin{enumerate}[1.]
 %\item This holds from properties of the convolution and Corollary \ref{conv}, also by using the fact that $C$ is not empty. 
 %\item 
%\end{enumerate}
\begin{itemize}
\item \textbf{\underline{A.8}} The submanifold $L_2\subset C\times C$ has finite codimension and furthermore the partial reduction $\underline{L_1}= L_1 / (L_2\cap L_1\times L_1)$ is a finite dimensional smooth manifold. 
We will denote $\underline{L_1}$ by $M$.
\item \textbf{\underline{A.9}} $S:= \{(c,[l]) \in C \times M: \exists l \in [l], g \in \mathcal G \vert (l,c,g) \in L_3\}$ is an immersed submanifold of $\mathcal G \times M $ satisfying 
the following three conditions:
\begin{enumerate}[1.]
\item
\begin{equation} \label{Source}
(S\times S)\circ L_2^{rel}= \triangle_M,
\end{equation}
where $L_2^{rel}: pt \nrightarrow C\times C$ is the induced relation from $L_2$.
\item 
The induced relation
\begin{equation}\label{subm}
dS:= TS: T\mathcal G \nrightarrow TM
\end{equation}
is surjective.
It is easy to check that the first condition implies the following 
\begin{corollary}\label{source}\emph{
\begin{enumerate}[1.]
\item The relation \[\, \,\,\,\,\,\,\,\,\,\,\;\;\,\;\;\;\;\;\;\;\;\;\,\;\;\,T:= \{(c,[l]) \in C \times M\colon  \exists l \in [l], g \in \mathcal G \vert (c,l,g) \in L_3\}=I\circ S\] is an immersed submanifold of $\mathcal G \times M$.
%\item $T= I \circ S$.
\item $S$ and $T$ regarded as relations from $C$ to $M$ are graphs of surjective submersions $s$ and $t$ respectively. 
\end{enumerate}
}
\end{corollary}
\begin{proof}
(1) follows from the cyclicity condition in A.1.  (2) follows by definition of $T$ and from the fact that, since Equation \ref{Source} holds, then if
$(x,[l])$ and $(x, [l]^{'})$ belong to $S\colon  C\to M$, then $([l], [l]^{'})\in \triangle_M$, which implies that $S$ is a surjective map and is clearly a submersion by Axiom 9, part 2.
\end{proof}
\item For any $f \in \mathcal C^{\infty}(M)$ the function $s^*f \in \mathcal C^{\infty}(C)$ is Hamiltonian with respect 
to the restriction of the symplectic form $\omega$ to $C$ \footnote{ This condition will be used to define a Poisson structure on $M$ and 
it is satisfied in all the examples of regular relational symplectic groupoids we have. }.
\end{enumerate}
\begin{remark}\emph{
The condition  given by Equation \ref{char} is at the level of tangent spaces. If  we want that $R^C=L_2$, we should impose a connectedness condition on the leaves of the characteristic foliation and the classes of $L_2$. The following is a modification of the example given in Remark \ref{counter1} of a structure satisfying all the axioms except the global version of Equation \ref{char}.
\begin{enumerate}[1.]
\item $\mathcal G = \mathbb Z$ 
\item $L=\{(n,m,-n-m-2k-1)\mid (m,n,k) \in \mathbb Z ^3\}$
\item $I: n \mapsto -n$
\end{enumerate}
For this example, the spaces $L_i$ and $C$ are given by}
\begin{eqnarray*}
L_1&=& \{ 2\mathbb Z +1\}\\
L_2&=& \{ (m,n)\mid m-n \in 2\mathbb Z\}\\
L_3&=&\{(m,n,m+n+2k+1)\mid n,m,k  \in \mathbb Z ^3 \}\\
C&=& \mathbb Z.
\end{eqnarray*}
\emph{Since $\mathbb Z$ is zero dimensional, we get that $C$ is also Lagrangian and for the symplectic reduction, $\underline{C}=*$. In the other hand,}
$$C / L_2= \mathbb Z/ 2 \mathbb Z \neq *.$$
\end{remark}
The following theorem connects the construction of relational symplectic groupoids in the regular case with the usual symplectic groupoids.
\begin{theorem}\label{TheoSym}\emph{ Let $(G,L,I)$ be a regular relational symplectic groupoid. Then $G:= C/L_2\rightrightarrows M$ is a topological groupoid over $M$. Moreover, if $G$ is a smooth manifold, then $G\rightrightarrows M$ is a symplectic groupoid over $M:= L_1/L_2$.
}
\end{theorem}
\begin{proof}
We denote by
\begin{eqnarray*}
p\colon  \mathcal G &\to& G\\
g&\mapsto& [g]
\end{eqnarray*}
the canonical projection with respect to the symplectic reduction of $C$. Then the following data
\begin{eqnarray*}
G_0&=& L_1/L_2 \\
G_1&=& C/L_2\\
G_2&=&C/L_2\times_{L_1/L_2} C/L_2
\\
m&=& p^3(L_3)\colon  G_2 \to G_1\\
\varepsilon&:& G_0 \to G_1:= \underline{\varepsilon}:L_1/L_2 \to C / L_2 \\
s&:&G_1 \to G_0:=\underline s\colon  C/L_2 \to M  \\
t&:&G_1 \to G_0:= \underline t\colon  C/L_2 \to M  \\
\iota&:& G_1 \to G_1:= \underline {I}\colon  C/L_2 \to C/L_2.
\end{eqnarray*}
corresponds to a groupoid structure.
Under the smoothness assumption for $\underline C$ and  by the finite dimensionality condition given in A.8 it follows that $G_1$ is a finite dimensional symplectic manifold and due to Corollary \ref{source}, the map $\underline s$ is a surjective submersion, hence, the fiber product $G_2$ is a  finite dimensional (topological) manifold. It is easy to check that the groupoid axioms are satisfied.
For the symplectic structure on $G\rightrightarrows M$, note that the projection of $L_3$ in $G$ is Lagrangian and restricted to $G_2$ is a map (due to Corollary \ref{source}). 

\end{proof}
\subsection{Poisson structure on $M$}
In this section, the goal is to relate the construction of the relational symplectic groupoids in the regular case  with 
Poisson structures in the space $M$. More precisely, we prove the existence and uniqueness of a Poisson bracket on 
$M$ compatible with a given regular relational symplectic groupoid $\mathcal G$. 
This theorem is the analog of the existence and uniqueness of a Poisson structure in the space of objects of usual 
symplectic groupoids \cite{Weinstein}. Namely,\begin{theorem}\label{Poisson base}
\emph{ \cite{Weinstein}. Let ${(G,\omega) \rightrightarrows M}$ be a symplectic groupoid over $M$ . Then there exists a unique Poisson structure $\Pi$ on $M$ such that the source map $s$ is a Poisson map (or equivalently the target map $t$ is an anti-Poisson map).}
\end{theorem}
One possible way to prove this theorem is by the use of what is known in the literature as \textit{Libermann's lemma}, that is stated in a slightly different formulation by Paulette Libermann in \cite{Libermann}, for the case of finite dimensional symplectic manifolds. 
Before stating the result, we need some definition that will be used in the sequel.
\begin{definition}
\emph{Let $(G, \omega)$ be a symplectic manifold and $\mathcal F$ a foliation on $G$. $\mathcal F$ is called a \emph{symplectically complete foliation} if the symplectically orthogonal distribution $(T \mathcal F)^{\perp}$ is an integrable distribution.
}
\end{definition} and equivalently, $\mathcal F$ is symplectically complete if there exists another foliation 
$\mathcal F^{'}$ such that $$T_x(\mathcal F)= (T_x(\mathcal F^{'}))^{\perp}.$$

After this definition, Libermann's lemma reads as follows

\begin{lemma}\label{Libermann} \emph{(Libermann). 
Let $\phi\colon  (G,\omega) \to M$ be a surjective submersion from a symplectic manifold $G$ to a manifold $M$ such that the fibers are connected. Denote by $(G,\mathcal F)$ the foliation on $G$ induced by the fibers of $\phi$.   Then, there exists a unique Poisson structure $\Pi$ on $M$ such that the map $\phi\colon  G \to M$ is a Poisson map if and only if the foliation $\mathcal F$ is symplectically complete.}
%Let $G$ be a symplectic manifold and $s, t \colon  G \to M$ be smooth surjective submersions of $G$ onto a smooth manifold $M$ such that the fibers $s^{-1}(x)$ and $t^{-1}(x)$ are symplectic orthogonal, for all $x \in M$. Then, there exists a (unique) Poisson structure $\Pi$ on $M$ such that the map $s$ is a Poisson map, or equivalently $t$ is an anti-Poisson map.}
\end{lemma}
\begin{proof}
Here we present a sketch of the proof. It can be checked that the distribution $(T\mathcal F)^{\perp}$ is Hamiltonian, 
i.e. is generated by the vector fields $X_{\phi^*f}$, for which $\omega(X_{\phi^*f}, \cdot)= d \phi^*f$, 
with $f \in \mathcal C^{\infty}(M)$. The fact that such distribution is integrable is equivalent, due to Frobenius Theorem, to the fact
that $[X_{\phi^*f}, X_{\phi^*g}]$, with $f,g \in \mathcal C^{\infty}(M)$, is tangent to the distribution $(T\mathcal F)^{\perp}$. This means that $X_{\{\phi^*f, \phi^*g\}}$ is tangent to 
$(T\mathcal F)^{\perp}$, therefore the bracket $\{\phi^*f, \phi^*g \}$ is constant along the $\phi-$ fibers and this implies that 
there exists a function $h\in\mathcal C^{\infty}(M)$ such that $\{\phi^*f, \phi^*g\}= \phi^*h$. This functions defines uniquely the Poisson bracket on 
$M$ given by 
$$\{f,g\}_{\Pi}=h.$$
\end{proof}

By applying Lemma \ref{Libermann} to the case of a symplectic groupoid $G \rightrightarrows M$ with $\mathcal F$ being the foliation described by the 
distribution $\ker(ds)$, Theorem \ref{Poisson base} holds.
%\begin{lemma} \emph{(Libermann). Let $G$ be a symplectic manifold and $s, t \colon  G \to M$ be smooth surjective submersions of $G$ onto a smooth manifold $M$ such that the fibers $s^{-1}(x)$ and $t^{-1}(x)$ are symplectic orthogonal, for all $x \in M$. Then, there exists a (unique) Poisson structure $\Pi$ on $M$ such that the map $s$ is a Poisson map, or equivalently $t$ is an anti-Poisson map.}
%\end{lemma}
The generalization of this result in the case of regular relational symplectic groupoids is the following Theorem, 
now in terms of Dirac structures.
\begin{theorem}\label{Theo1Poisson1}\emph{ Let $(\mathcal G, L, I)$ be a regular relational symplectic groupoid, with $M= L_1/ L_2$.  Then, assuming that the $s$-fibers are connected, there exists a unique Poisson structure $\Pi$ on $M$ such that the map $s\colon  C\to M$ is a forward-Dirac map (or equivalently, $t\colon  C\to M$ is an anti forward-Dirac map).
}
\end{theorem}
\begin{proof}
For this proof we present a more general version of Libermann's lemma for the case when $G$ is presymplectic. 
\begin{definition}\emph{ Let $M$ be a Banach manifold. A $2-$form $\omega \in \Omega^2(\mathcal G)$ is called \emph{presymplectic} if $d\omega=0$. In this case $M$ is called a presymplectic manifold. 
}
\end{definition}
\begin{definition}
\emph{ Let $M$ be a presymplectic manifold. A function $f \in \mathcal C^{\infty}(M)$ is called \emph{Hamiltonian} if there exists a vector field $X_f \in \mathfrak X(M)$ such that 
$$\iota_{X_f} \omega= df.$$
}
\end{definition}
Given this definition, we recall the following basic facts about presymplectic manifolds
\begin{enumerate}[1.]
\item If two functions  $f$ and $g$ are Hamiltonian, then the function
$$\{f,g\}_{M}:= \iota_{X_f}\iota_{X_g} \omega$$ is also Hamiltonian.
\item The bracket 
\begin{eqnarray*}
\{,\}: \mathcal C^{\infty}(M)\otimes \mathcal C^{\infty}(M) &\to& \mathcal C^{\infty}(M)\\
f\otimes g &\mapsto& \{f,g\}_M  
\end{eqnarray*}
is Poisson.
\item If $f$ and $g$ are Hamiltonian, the following equation holds
$$d \{f,g\}_{M}= \iota_{[X_f,X_g]} \omega.$$
\end{enumerate}
With this properties at hand, we are able to state the extension of Libermann's lemma in the case of presymplectic manifolds.
\begin{lemma}
\emph{(Libermann's lemma for presymplectic manifolds). Let $\mathcal G$ be a presymplectic manifold and 
$s, t : \mathcal G \to M$ be smooth surjective submersions of $\mathcal G$ onto a smooth manifold $M$ such that the fibers $s^{-1}(x)$ and $t^{-1}(x)$ are connected and mutually presymplectic orthogonal, for all $x \in M$. 
Assume that, for all $f \in \mathcal C^{\infty}(M)$ the functions $s^*f$ and $t^*g$ are Hamiltonian. Then, there exists a unique Poisson structure $\Pi$ on $M$ such that the map $s$ is a forward-Dirac map.}
\end{lemma} 
\begin{proof}
The idea of the proof is quite similar to the one for Lemma \ref{Libermann}. We will prove that $s^* \mathcal C^{\infty}(M)$ is a Poisson subalgebra of $\mathcal C^{\infty}(M)_{Ham}$, where $\mathcal C^{\infty}(M)_{Ham}$ denotes the Poisson algebra of Hamiltonian functions in $\mathcal C^{\infty}(M)$. \\
Since $s^*f$ and $s^*g$ are Hamiltonian, for all 
$f,g$ in $\mathcal C^{\infty}(M)$, then there exist vector fields $X_{s^*f}$ and $X_{s^*g}$ in $\Gamma(G)$ such that
$$\omega(X_{s^*f}, \cdot)= d s^*f$$
and 
$$\omega(X_{s^*g}, \cdot)= d s^*g.$$
The bracket between $f$ and $g$ is defined as follows.Denoting $\omega(X_{s^*f}, X_{s^*g})$ by $\{s^*f,s^*g\}_{\mathcal G}$, 
it follows that

%Now, let $W$ be a subspace of $\Gamma(TG)$ and we define pointwise
%$$W_x^{\omega}:= \{v \in T_xG \mid \omega(v,w)=0, \forall w \in W_x \}$$
%and we also define
%$$W^0= \{\xi \in T^*G \mid \xi(V)=0, \forall V \in T_xG.\}$$
%Now, let $W= \ker(t_*) \subset T_xG$ and $Y \in W$. We have that 
%$$\xi:= \omega(Y,\cdot) \in \mathfrak{Im}(\omega^{\sharp}) \cap (\ker(t_*))^0$$
%Since the $t-$ fibers are connected, we could write $\xi$ as
%\begin{equation}
%\xi= \sum_i \alpha_i dh_i,
%\end{equation}
%with $h_i \in \mathcal C^{\infty}(M)$.
%Now, given two functions $f$ and $g$ in $\mathcal C^{\infty}(M)$, we construct the bracket $\{f,g\}_{M}$ as follows. Since 
$$Y\{s^*f, s^*g\}_{\mathcal G}=0,\, \forall Y \in \Gamma(\ker (s_*)),$$
This implies that $\{s^*f, s^*g\}_{\mathcal G}$ is constant along the $t$-fibers, hence, there exists a unique $h \in \mathcal C^{\infty}(M)$ 
such that $\{s^*f, s^*g\}_{\mathcal G}= t^*h$. We then define the Poisson bracket between $f$ and $g$ as
$$\{f,g\}_{\Pi}:=h.$$
%then, from the discussion above, there exists $\alpha \in \Gamma(\ker t_*)^0 \cap \mathfrak{Im}(\omega^{\sharp})$ 
%such that $Y= (\omega^{\sharp})^{-1}(\alpha)$. Since $Y$ can be written as
%$$ Y= \sum_i X_{t^*h_i}, \, h_i \in \mathcal C^{\infty}(M),$$ we can define
%\begin{equation}
%\{ f,g \}_{M}:=h= \sum_i \alpha_i h_i.
%\end{equation}
The fact that $s$ is a forward-Dirac map with respect to $\{,\}_M$ is equivalent to the following equation
\begin{equation*}
\{t^*h, \{s^*f, s^*g \}\}_{G}=0
\end{equation*}
which holds since $\{,\}_G$ satisfies the Jacobi identity, and in a similar way as in Lemma \ref{Libermann}, 
it can be checked that $\{f,g\}_{\Pi}$ is then a Poisson bracket.
\end{proof}
We can apply this lemma for the case when $\mathcal G$ is the weak symplectic manifold 
of a regular relational symplectic groupoid, $M$ is the quotient $L_1/ L_2$ and $s$ and $t$ as defined in Corollary \ref{source}, 
and this finishes the proof of the Theorem.

\end{proof}

%\begin{theorem} 
%\begin{enumerate}[1.]
 %\item There exists a unique Poisson structure on $M$ such that $S$ is coisotropic \footnote{The notion of coisotropic submanifolds at the infinite dimensional 
%level has to be worked out} in $\mathcal G \times M.$
%\item This is also the unique Poisson structure on $M$ such that $T$ is coisotropic in $\mathcal G \times M.$
%\end{enumerate}
%\end{theorem}

\end{itemize}
%Conjecturally, the proof of this theorem could be adapted in order to drop the connectedness assumption of the $s$-fiber. Then, we have the following
%\begin{conjecture}\label{Theo1PoissonC}\emph{
%Theorem \ref{Theo1Poisson1} holds also when the $s$-fibers are not connected.}
%\end{conjecture} 
\subsection{Equivalence in the regular case}
It follows from Remark \ref{equiv} that, in the case where $\mathcal G$ and $\mathcal H$ are both regular, an equivalence $F$ induces relations 
$F_M\colon M_{\mathcal G} \nrightarrow M_{\mathcal H}$ and $F_M^{\dagger}\colon M_{\mathcal H} \nrightarrow M_{\mathcal G}$ satisfying
\begin{eqnarray*}
F_M^{\dagger}\circ F_M&=& \Id_{M_{\mathcal G}}\\ 
F_M\circ F_M^{\dagger}&=& \Id_{M_{\mathcal H}}.
\end{eqnarray*}
This implies the following
\begin{lemma}
\emph{The induced relation $F_M$ is the graph of a diffeomorphism between $M_{\mathcal G}$ and $M_{\mathcal{H}}$.}
\end{lemma}

Since the following diagram commutes:
\[
\xymatrixrowsep{2pc} \xymatrixcolsep{2pc} \xymatrix{C_{\mathcal G} \ar@{->}[r]^{F} |-{/}\ar@{->}[d]_{s}&C_{\mathcal H}\ar@{->}[d]_{s}\\M_{\mathcal G}\ar@{->}[r]^{F_M} &
M_{\mathcal H}}
\]
from Theorem \ref{Theo1Poisson1} it follows  also the following
\begin{lemma}\emph{
The map $F_M$ is a Poisson diffeomorphism.}
\end{lemma}
By a similar argument it can be easily checked that 
\begin{lemma}\label{Poi}
\emph{If $\mathcal{G}$ and $\mathcal{H}$ are two equivalent regular relational symplectic groupoids, and if we assume that $\mathcal G$ has coonected $s_{\mathcal G}$- fibers , then there exists a unique Poisson structure in $M_{\mathcal H}$ such that the map $s_{\mathcal H}$ is forward-Dirac.}
\end{lemma}

\section{Examples of regular relational simplectic groupoids}

\subsection {Symplectic groupoids} \label{groupo}
Given a symplectic groupoid $G$ over $M$, we can endow it naturally with a relational symplectic structure:
\begin{eqnarray*}
\mathcal{G}&=&G.\\
%L_1&=&\varepsilon(M).\\
%L_2&=&\triangle_2(G).\\ 
L&=&\{(g_1,g_2,g_3) \vert (g_1,g_2) \in G\times_{(s,t)}G,\, g_3^{-1}=\mu(g_1,g_2)\}.\\
%S &=&\{(g,x)\in G \times M \vert s(g)=x \}.\\
%T &=&\{(g,x)\in G \times M \vert t(g)=x \}.\\
I &=& g \mapsto \iota(g),\, g \in G.
\end{eqnarray*}
In this case, it is an easy check that the immersed canonical relations $L_i$ are given by
\begin{eqnarray*}
L_1&=& \varepsilon(M)\\
L_2&=& \triangle(G)\\
L_3&=& Gr(\mu),
\end{eqnarray*} 
and also we observe that in this case $(\mathcal G, L, I)$ is regular.
%\end{example}
According to Theorem \ref{TheoSym}, given a regular relational symplectic groupoid $(\mathcal G, L,I)$ which admits a smooth symplectic reduction, we can associate a usual symplectic groupoid $G\rightrightarrows M$. By definition of such groupoid, we obtain the following
\begin{proposition}\label{pro}\emph{
The projection $p\colon \mathcal G \to G$ is an equivalence of relational symplectic groupoids.}
\end{proposition}
\begin{proof}
By Proposition \ref{projection}, $p$ is a canonical immersed relation $p\colon  \mathcal G \nrightarrow G$ and by definition of the groupoid structure on $G$ (see Theorem \ref{TheoSym}), it follows that $p$ commutes with $I$ and $L$ respectively, hence, $p$ is a morphism of relational symplectic groupoids.
The fact that $p^{\dagger}$ is also a morphism follows from the following facts. By definition, $I_G$ can be written as 
\begin{equation*}
I_G= p\circ I_{\mathcal G} \circ p^{\dagger}
\end{equation*}
and $L_G$ can be written as
\begin{equation*}
p\circ L \circ (p^{\dagger} \times p^{\dagger}).
\end{equation*}
Again, by proposition \ref{projection} we have that
\begin{equation} \label{Q1}
p^{\dagger} \circ p=(L_2)_{\mathcal G}
\end{equation}
and that
\begin{equation} \label{Q2}
p\circ p^{\dagger}= Id_G.
\end{equation}
Therefore, we get the following equalities
\begin{eqnarray}
p^{\dagger}\circ I_G&=& p^{\dagger}\circ p \circ I_{\mathcal G}\circ p^{\dagger}\\
&\stackrel{\ref{Q1}}{=}& L_2
\end{eqnarray}

\end{proof}
Another finite dimensional example is the following.
\subsection{Symplectic manifolds with a Lagrangian submanifold} Let $(G,\omega)$ be a symplectic manifold, $\phi$ an antisymplectomorphism and $\mathcal L$ an immersed 
Lagrangian submanifold of $G$ such that $\phi(\mathcal L)=\mathcal L$. We define
\begin{eqnarray*}
\mathcal{G}&=&G.\label{p}\\
%M&=& pt.\\
%L_1&=&L.\\
%L_2&=&L\times L.\\ 
L&=&\mathcal L \times \mathcal L\times \mathcal L.\label{q}\\
%S &=&\{(g,pt)\vert g\in G \}.\\
%T &=&\{(g,pt)\vert g\in G \}.\\
I&=& \phi\label{r}
\end{eqnarray*}
It is an easy check that this construction satisfies the relational axioms and that the spaces $L_i$ are given by 

\begin{eqnarray*}
L_1&=& \mathcal L \\
L_2&=& \mathcal L \times \mathcal L\\
L_3&=&\mathcal L\times  \mathcal L\times  \mathcal L.
\end{eqnarray*}

This example is a regular relational symplectic groupoid and furthermore we can prove the following
\begin{proposition} \emph{
The previous relational symplectic groupoid is equivalent to the zero dimensional symplectic groupoid (a point with zero symplectic structure and empty relations).}
\end{proposition}
\begin{proof}We prove that $\mathcal L$ is an equivalence from the zero dimensional manifold $p$ to $\mathcal{G}$ . This comes from the fact that, for this example, $C$ (as defined in Equation \ref{C}) is precisely $\mathcal L$, hence, its symplectic reduction is just a point. By Proposition \ref{pro} it follows that $\mathcal L$, being the canonical projection, is an equivalence.
\end{proof}
\begin{remark}\emph{More generally, following the definition of equivalence of relational symplectic groupoids and remark \ref{equiv}, we can check that $(\mathcal G,L,I)$ 
is equivalent to the zero dimensional symplectic groupoid if and only if there exists a Lagrangian submanifold $\mathcal L_{eq}$ of $\mathcal G$ satisfying the following two properties
\begin{itemize}
 \item $I\circ \mathcal L_{eq}=\mathcal L_{eq}$.
\item $\mathcal L_{eq}= L\circ (\mathcal L_{eq} \times \mathcal L_{eq})$. 
\end{itemize}
This implies that the only relational symplectic groupoids that are equivalent to the zero dimensional one are the ones described by Equations \ref{p}, \ref{q} and \ref{r}.}
 
\end{remark}

\subsection{Powers of symplectic groupoids}\label{power}
The following are two (a priori) different constructions of relational symplectic groupoid for the powers of a given symplectic groupoids. 
Let $G\rightrightarrows M$ be a symplectic groupoid and $(G,L,I)$ its associated relational symplectic groupoid as in Example \ref{groupo}. It is easy to check that
\begin{proposition}
\emph{
$(G^n, L^n, I^n)$ is a regular relational symplectic groupoid, for all $n\geq 1$.}
\end{proposition}
In this case the base Poisson manifold is $M^n$, with $M$ the base Poisson manifold of $G$. 
Now, let us denote $G_{(1)}=G$, $G_{(2)}$ the fiber product $G\times_{(s,t)}G$, $G_{(3)}=G \times_{(s,t)}(G \times_{(s,t)}G)$ and so on. We will use the following
\begin{lemma} \emph{\cite{Xu}. Let $G\rightrightarrows M$ be a symplectic groupoid.
\begin{enumerate}[1.]
 \item $G_{(n)}$ is a coisotropic submanifold of $G^n$.
\item The reduced spaces $\underline{G_{(n)}}$ are canonically symplectomorphic to $G$ (by the first component projection). Furthermore, there exists a natural symplectic groupoid structure on $\underline{G_{(n)}}\rightrightarrows M$ coming from the symplectic quotient, isomorphic to the symplectic groupoid structure on $G\rightrightarrows M$. 
\end{enumerate}
}
\end{lemma}
Having this lemma at hand and considering the canonical relations 
$$p_n: G^n \nrightarrow \underline{G_{(n)}}\equiv G,$$ we define the regular relational symplectic groupoid $(G^{(n)}, L^{(n)}, I^{(n)})$, given by 
\begin{eqnarray*}
G^{(n)}&:=& G^n\\
I^{(n)}&:=& p_n^{\dagger} \circ I \circ p_n\\
L^{(n)}&:=& p^{\dagger} \circ L \circ (p \times p),
\end{eqnarray*}
where for each $n$, $(G^{(n)}, L^{(n)}, I^{(n)})$ is equivalent, as relational symplectic groupoid, to $(G, L,I)$.
% Definition of regular rsg; Poisson structure on objects; effect of morphisms on objects; equivalence; $L_2$ as an equivalence
 % Examples of rrsg: sg; fiber product of sg; symplecctic manifold with Lag submanifold
  %Thm: any two sg's for the same M are equivalent

\section{RSG and integration of Poisson manifolds}
\subsection{Symplectic groupoids as the phase space of the PSM}
We introduce the Poisson sigma model
associated to a Poisson manifold
(the classical version of the model through the Hamiltonian
formalism). After the construction of the reduced phase space of the PSM
associated to integrable Poisson manifolds we generalize
the construction to the non reduced version by defining the
\emph{relational symplectic groupoid}.
We give examples and we concentrate our attention on the examples
coming from Poisson geometry.

\begin{definition} 
\emph{
A Poisson sigma model (PSM) corresponds to the following data:
\begin{enumerate}[1.]
\item A compact surface $\Sigma$, possibly with boundary, called the \textit{source}.
\item A finite dimensional Poisson manifold $(M,\Pi)$, called the \textit{target}.
\end{enumerate}
}

%Recall that a bivector field $\Pi \in \Gamma(TM \wedge TM) $ 
%is called Poisson if the the bracket $\{,\}\colon \mathcal{C}^{\infty}(M) \otimes \mathcal{C}^{\infty}(M)\to \mathcal{C}^{\infty}(M)$, defined by
%\[\{f,g\}= \Pi(df,dg)\]
%is a Lie bracket and it satisfies the Leibniz identity
%\[\{f,gh\}=g \{f,h\}+h\{f,g\}, \forall f,g,h \in \mathcal{C}^{\infty}(M).\] 
%In local coordinates, the condition of a bivector $\Pi$ to be Poisson reads as follows
%\begin{equation}\label{SN}
%\Pi^{sr}(x)(\partial_r) \Pi^{lk}(x)+\Pi^{kr}(x)(\partial_r) \Pi^{sl}(x)+\Pi^{lr}(x)(\partial_r) \Pi^{ks}(x)=0,
%\end{equation}

%that is, the vanishing condition for the Schouten-Nijenhuis bracket of $\Pi.$ 

\end{definition}
The space of fields for this theory is denoted with $\Phi$ and corresponds to the space of vector bundle morphisms of class $\mathcal C^{k+1}$ between $T\Sigma$ and $T^*M$.
This space can be parametrized by a pair $(X, \eta)$, where $X\in \mathcal C^{k+1}(\Sigma, M)$  and $\eta \in 
\Gamma^k(\Sigma, T^*\Sigma \otimes X^*T^*M),$ and $k \in \{0,\,1,\, \cdots \, , \infty\}$ denotes the regularity type of the map that we choose to work with.\\
On $\Phi$, the following first order action  is defined:
\[S(X,\eta):= \int_{\Sigma} \langle \eta,\, dX\rangle+  \frac 1 2 \langle \eta, \, (\Pi^{\#}\circ X) \eta \rangle,\]
where,

 \begin{eqnarray*}
 \Pi^{\#}\colon T^*M&\to& TM\\
  \psi &\mapsto& \Pi(\psi, \cdot).
 \end{eqnarray*}
 
% is the map from $T^*M\to TM$ induced from the Poisson bivector $\Pi$.

Here,  $dX$ and $\eta$ are regarded as elements in $\Omega^1(\Sigma, X^*(TM))$, $\Omega^1(\Sigma, X^*(T^*M))$, respectively and  $\langle \,,\, \rangle $ is the pairing between $\Omega^1(\Sigma,  X^*(TM))$ and $\Omega^1(\Sigma,  X^*(T^*M))$ induced by the 
natural pairing between $T_xM$ and $T_x^*M$, for all $x \in M$.\\
\begin{remark}\emph{
This model has significant importance for deformation quantization. Namely, the perturbative expansion of the Feynman path integral 
for the PSM, in the case that $\Sigma$ is a disc, gives rise to Kontsevich's star product formula 
\cite{CattaneoDeformation, Cat, Kontsevich1}, i.e. the semiclassical expansion of the path integral
\begin{equation*}
\int_{X(r)=x}f(X(p)) g (X(q))\exp{(\frac{i}{\hbar} S(X,\eta)) dX d\eta}
\end{equation*}
around the critical point $X(u)\equiv x, \eta \equiv 0$, where $p,q$ and $r$ 
are three distinct points of $\partial \Sigma$, corresponds to the star product $f\star g (x)$.}
\end{remark}

\section{The PSM and its phase space}
For this model, we consider the constraint equations and the space of gauge symmetries. These will allow us to understand the geometry of the phase space and its reduction. First, we define
\[\mbox{EL}_{\Sigma}=\{\mbox{Solutions of the Euler-Lagrange equations}\}\subset \Phi,\]
where, using integration by parts \[\delta S= \int_{\Sigma} \frac{\delta \mathcal{L}}{\delta X} \delta X+ \frac{\delta \mathcal{L}}{\delta \eta} \delta \eta + \mbox{boundary terms}.\]The partial variations correspond to:
\begin{eqnarray}
\frac{\delta \mathcal{L}}{\delta X}&=& dX+ \Pi^{\#}(X)\eta=0\\
\frac{\delta \mathcal{L}}{\delta \eta}&=& d\eta+ \frac 1 2 \partial \Pi^{\#}(X)\eta\wedge\eta=0.
\end{eqnarray}

Now, if we restrict to the boundary, the general space of boundary fields corresponds to

\[\Phi_{\partial}:=\{\mbox{vector bundle morphisms between  } T (\partial \Sigma) \mbox{  and  } T^*M\}.\]
 
Following \cite{Cat, AlbertoPavel1}, $\Phi_{\partial}$ is endowed with a weak symplectic form and a surjective submersion $p\colon  \Phi \to \Phi_{\partial}$. 
Explicitly we have the following description of the space of boundary fields:
 \begin{remark} \emph{ For $k < \infty$, $\Phi_{\partial}$ can be identified with the Banach manifold $$P(T^*M):= \mathcal C^{k+1}(I, T^*M),$$ therefore it is a weak symplectic Banach manifold locally modeled by $\mathcal C^{k+1}(I, \mathbb R ^{2n})$.  
 } 
 \end{remark}
In order to see this, we understand $\Phi_{\partial}$ as a \emph{fiber bundle} over the path space $PM$, which is naturally equipped with the topology of uniform convergence. The fibers of the bundle are isomorphic to the Banach space of class $\mathcal C^{k}$
\begin{equation*}
T^*_X(PM):= \Omega^1(I, X^*(T^*M))
\end{equation*}
Therefore, as a set, $\Phi_{\partial}$ corresponds to
\begin{equation*}
\Phi_{\partial}= \bigcup_{(X \in PM)} T^*_X(PM).
\end{equation*}
The identification with $P(T^*M)$ is explicitly given by
\begin{eqnarray*}
\psi\colon  T^*(PM) &\to& P(T^*M)\\
     (X,\eta) &\mapsto& (\gamma\colon  t \mapsto (X(t),\eta(t)))
\end{eqnarray*}
and this allows to define a 2-form $\omega$ in $\phi_{\partial}$ in the following way. 
Identifying the tangent space $T_{\gamma} (P(T^*(M)))$ with the space of vector fields along the curve $\gamma$
\begin{equation*}
T_{\gamma} (P(T^*(M)))= \{\delta\gamma: I \to TT^*M \mid \delta \gamma(t) \in T_{\gamma(t)}T^*M \}
\end{equation*}
the 2-form $\omega$ in $\phi_{\partial}$ is given by
\begin{equation*}
\omega_{\gamma}(\delta_1\gamma, \delta_2\gamma)= \int_0^1 \omega^{Liouv}(\delta_1\gamma(t),\delta_2 \gamma(t)) dt,
\end{equation*}
where $\omega^{Liouv}= d \alpha^{Liouv}$ is the canonical symplectic form on $T^*M$. In local coordinates, if  $\gamma$ is 
described by the functions $X^1(t), X^2(t), \cdots X^n(t) \in \mathcal C^{k+1}(I)$ and $\eta_1,\eta_2,\cdots \eta_n \in \Omega^1(I)$ of class $\mathcal C^{k}$, 
then $\omega$ is given by
\begin{equation}\label{symplectic}
\omega_{\gamma}(\delta_1\gamma, \delta_2\gamma)= \int_0^1(\delta_1X^i(t)\delta_2\eta_i(t)-\delta_2X^i(t)\delta_1\eta_i(t))dt.
\end{equation}
This form is clearly closed since
\begin{equation*}
d\omega_{\gamma}(\delta_1\gamma, \delta_2\gamma,\delta_3\gamma)= \int_0^1 d\omega^{Liouv}(\delta_1\gamma(t),\delta_2 \gamma(t), \delta_3\gamma(t)) dt=0.
\end{equation*}
It is weak symplectic since, if $\omega^{\sharp}(\delta_1 \gamma)=\omega^{\sharp}(\delta_1^{'} \gamma)$, 
then, we can set $$\delta_1\eta_i\equiv 0, \forall 1 \leq i \leq n ,$$ in this case
\begin{equation*}
\int_0^1 (\delta_1X^i(t)-\delta_1(X^i)^{'}(t))\delta_2\eta_i(t)=0, \forall \delta_2\eta_i(t),
\end{equation*}
which implies that $$\delta_1X^i(t)=\delta_1(X^i)^{'}(t).$$ 
If we set now $$\delta_2\eta_i\equiv 0, \forall 1 \leq i \leq n,$$ 
we can conclude in a similar way that $\delta_1\eta_i(t)=\delta_1(\eta_i)^{'}(t).$

Now, we define, following \cite{AlbertoPavel1}
$$L_{\Sigma}:= p(EL_{\Sigma}).$$
%where $p: \phi \to C_{\partial \Sigma}$ is a surjective submersion and $ C_{\partial \Sigma}$ denotes the space of Cauchy data of the PSM restricted to the boundary $\partial \Sigma$.

Finally, we define $C_{\partial}$ as the set of fields in $\Phi_{\partial}$ which  can be completed to a field in  $L_{\Sigma^{'}}$, with $\Sigma^{'}:= \partial \Sigma \times [0, \varepsilon]$, for some $\varepsilon$.  \\
%the constraint to the boundary is defined as
 %It turns out that $\Phi_{\partial}$ can be identified with $T^*(PM)$, the cotangent bundle of the path space on $M$ and that
 It can be proven that 
 \begin{proposition}\label{Coiso} \emph{\cite{Cat}.
 \begin{enumerate}[1.]
 \item The space $C_{\partial}$ is described by
 \begin{equation}\label{Cois}C_{\partial}= \{(X,\eta)\vert dX= \pi^{\#}(X)\eta,\, X\colon  \partial \Sigma \to M, \, \eta \in \Gamma (T^*I \otimes X^*(T^*M))\}.\end{equation}
%\[L(\partial \Sigma)=\{\mbox{Lie algebroid morphisms between  } T([0,1]) \mbox{  and  } T^*M .\}\ proven that $L(\partial \Sigma)$ is a coisotropic submanifold of $\Phi_{\partial}$ \footnote{see proof in the Appendix.} and in \cite{Cat} the following Theorem is proven:
\item The space $C_{\partial}$ is a coisotropic Banach submanifold of $\Phi_{\partial}$ and its associated characteristic foliation has codimension $2n$, where $n = \emph{dim} (M)$.
\end{enumerate}
}
\end{proposition}
In fact, the converse of the second property also holds in the following sense. If we define $S(X,\eta)$ and $C_{\partial}$ 
in the same way as before, it can be proven that

\begin{proposition}\label{Coi}\emph{\cite{Coiso}.
If $C_{\partial}$ is a coisotropic submanifold of $\Phi_{\partial}$, then $\pi$ is a Poisson bivector field.}
\end{proposition}
The following geometric interpretation of $\mathcal C_{\partial}$   will lead us to the connection between Lie algebroids and Lie groupoids in Poisson geometry with the PSM. 
The condition for a vector bundle morphism to preserve the Lie algebroid structure gives rise to some PDE's that the anchor maps and the structure functions for $\Gamma(A)$ and  $\Gamma(B)$ should satisfy. For the case of PSM, regarding $T^*M$ as a Lie algebroid, we can prove that 

\[C_{\partial}=\{\mbox{Lie algebroid morphisms between  } T (\partial \Sigma) \mbox{  and  } T^*M\},\]
where the Lie algebroid structure on the left is given by the Lie bracket of vector fields on $T (\partial \Sigma)$ with identity anchor map.

\section{Symplectic reduction}
Since $C_{\partial}$ is a coisotropic submanifold, it is possible to perform symplectic reduction, which yields, when it is smooth, a symplectic finite dimensional manifold. In the case of $\Sigma$ being a rectangle and with vanishing boundary conditions for $\eta$ (see \cite{Cat}), following the notation in \cite{Crai} and \cite{Severa}, we could also reinterpret the reduced phase space $\underline{C_{\partial}}$ as 

\[\underline{C_{\partial}}=\left \{ \frac{\mbox{$T^*M$-paths}}{ T^*M \mbox{-homotopy}} \right \}.\]
In the integrable case, it was proven in \cite{Cat} that

\begin{theorem} The following data
\begin{eqnarray*}
G_0&=& M \\
G_1&=& \underline{C_{\partial}}\\
G_2&=&\{ [X_1, \eta_1], [X_2, \eta_2] \vert X_1(1)=X_2(0)\}
\\
m&\colon & G_2 \to G:= ([X_1, \eta_1], [X_2, \eta_2]) \mapsto [(X_1\ast X_2, \eta_1\ast \eta_2)] \\
\varepsilon&:& G_0 \to G_1:= x\mapsto [X\equiv x, \eta\equiv 0] \\
s&\colon &G_1 \to G_0:= [X, \eta]\mapsto X(0) \\
t&\colon &G_1 \to G_0:= [X, \eta]\mapsto X(1) \\
\iota&:& G_1 \to G_1:= [X, \eta] \to [i^* \circ X,i^* \circ \eta]\\
&\mbox{                   }& i\colon [0,1]\to [0, 1]:= t\to 1-t, 
\end{eqnarray*}
defines a symplectic groupoid that integrates the Lie algebroid $T^*M$.  \footnote{here $\ast$ denotes path concatenation}
\end{theorem}
\begin{remark}\emph{In \cite{Cat}, this construction is also expressed as the Marsden-Weinstein reduction of the Hamiltonian action  of the (infinite dimensional) Lie algebra
$P_0\Omega^1(M):= \{\beta \in  \mathcal C^{k+1}(I, \Omega^1(M)) \mid \beta(0)=\beta(1)=0\}$
with Lie bracket
\begin{equation*}
[\beta,\gamma](u)=d\langle \beta(u), \Pi^{\sharp} \gamma(u) \rangle- \iota_{\Pi^{\sharp}(\beta(u))} d\gamma(u)+ \iota_{\Pi^{\sharp}(\gamma(u))} d\beta(u)
\end{equation*}
on the space $T^*(PM)$, on which the moment map $ \mu\colon  T^*(PM) \to P_0\Omega^1(M)^*$  is described by the equation 
\begin{equation*}
\langle \mu(X,\eta), \beta \rangle= \int_0^1 \langle dX(u)+ \Pi^{\sharp}(X(u))\eta(u), \beta(X(u),u) \rangle du.
\end{equation*}}
\end{remark}

\section{Integration of Poisson manifolds via the PSM}
The goal of this Section is to prove the following Theorem
\begin{theorem}\label{Relational} Given a Poisson manifold $(M, \Pi)$ there exists a regular relational symplectic groupoid $(\mathcal G,L,I)$ that integrates it.
\end{theorem}
As we mentioned in the Introduction, integration in this setting means the following 
\begin{enumerate}[1.]
\item Such relational symplectic groupoid satisfies that $L_1/ L_2=M$ and the symplectic structure on $\mathcal G$ is compatible with the Poisson structure on $M$ according to Theorem \ref{Theo1Poisson1}
\item In the case that the Lie algebroid $T^*M$ is integrable, such relational symplectic groupoid is equivalent to a symplectic groupoid integrating it.
\end{enumerate}
The structure of the proof of this Theorem is as follows. First, we describe the defining data for the relational symplectic groupoid in terms of the PSM and $A-$ homotopy for Lie algebroids specialized in the Poisson case. Then we verify that such data in fact satisfy the relational axioms. In order to do this, we need to prove the smoothness and the Lagrangian property of the canonical relations $L_i$, which deserves special attention since we are dealing with infinite dimensional spaces.\\

\textbf{Proof of Theorem \ref{Relational}}
We will prove that the relational symplectic groupoid $(\mathcal G, L, I)$ associated to $(M, \Pi)$ is given by 
\begin{enumerate}[1.]
\item $\mathcal G := T^*(PM)$, the cotangent bundle of the path space of $M$.
\item $L=\{(\gamma_1, \gamma_2, \gamma_3) \in (T^*(PM))\}$ is such that
\begin{itemize}
\item $\gamma_i$, with $1 \leq i \leq 3$ are $T^*M$-paths.
\item The concatentation $\gamma_1 \ast \gamma_2$ is $T^*M$- homotopic to the inverse path $\gamma_3^{-1}$, or equivalently, $\gamma_1\ast \gamma_2 \ast \gamma_3$ is $T^*M$-homotopic to a constant path \footnote{$\gamma^{-1}$ denotes here $\iota \circ \gamma$ }.
\end{itemize}
\item \begin{eqnarray*}
I\colon  T^* PM &\to& T^* PM\\
\gamma &\mapsto& \gamma^{-1}.
\end{eqnarray*}

\end{enumerate}

First, we describe the defining spaces $L_i$ of the relational symplectic groupoid 
set theorically, proving that they satisfy the algebraic relational axioms and then we prove that they are in fact immersed canonical relations.\\

\textbf{\underline{A.1.}} To prove the cyclicity property, we use the following remark, that is easy to check.
\begin{remark}
\emph{ Let $\gamma_1, \, \gamma_2, \, \gamma_1^{'}$ and $\gamma_2^{'}$ be $T^*M$- paths such that $\gamma_1 \sim \gamma_1^{'}$ and $\gamma_2 \sim \gamma_2^{'}$, where $\sim$ denotes the equivalence by $T^*M$- homotopy. Then 
$$\gamma_1 * \gamma_2 \sim \gamma_1^{'} * \gamma_2^{'}.$$ 
}
\end{remark}
Now, consider $(x,y,z) \in L$. Since $x * y \sim z^{-1}$, we get that
\begin{equation*}
x * y \sim z^{-1}\Leftrightarrow (x\ast y)\ast y^{-1} \sim z^{-1} * y^{-1} \Leftrightarrow z\ast (x\ast y)\ast y^{-1} \sim z\ast z^{-1} \ast y^{-1}\Leftrightarrow z \ast x\sim y^{-1},
\end{equation*}
hence, $(z,x,y)$ (and similarly $(y,z,x)$) belongs to $L$.
\qed
\\

\textbf{\underline{A.2.}} 
If we define
\begin{eqnarray}\label{phi}
\phi\colon  [0,1] &\to& [0,1]\\
t&\mapsto& 1-t
\end{eqnarray}
Then we get that
\begin{eqnarray*}
I\colon  T^*(PM) &\to& T^*(PM)\\
\gamma &\mapsto& \phi^*\circ \gamma,
\end{eqnarray*}
hence,
$$I_{*}\delta \gamma =I_{*}(\delta X, \delta \eta)=\delta X(\phi(t)), -\delta \eta(t))$$
and therefore, using Equation \ref{symplectic},
\begin{equation*}
I^*\omega_{\gamma}(\delta_1\gamma, \delta_2\gamma)= -\int_0^1\delta_1X^i(t)\delta_2\eta_i(t)-\delta_2X^i(t)\delta_1\eta_i(t)dt=-\omega_{\gamma}(\delta_1\gamma, \delta_2\gamma)
\end{equation*}
and this proves that $I$ is an anti-symplectomorphism.
\qed
\\

\textbf{\underline{A.3.}} First, we observe that, from the definition,
\begin{equation}\label{ELE3}L_3=\{(\gamma_1, \gamma_2, \gamma_3)\in T^*(PM)^3 \mid \gamma_1* \gamma_2 \sim \gamma_3 \}.
\end{equation}
In Subsection \ref{smoothness} we will prove that $L_3$ is an immersed canonical relation.
\\

\textbf{\underline{A.4.}} We have that 
\begin{eqnarray*}
L_3\circ (L_3\times Id)&=& \{(\gamma_1, \gamma_2,\gamma_3, \gamma_4) \in (T^*(PM))^3 \mid \exists (\gamma_5, \gamma_6)\in T^*(PM)^2\\
 &\mid& (\gamma_1, \gamma_2, \gamma_5)\in L_3, (\gamma_3, \gamma_6) \in Id, (\gamma_5, \gamma_6, \gamma_4) \in L_3.\}
\end{eqnarray*}
Given the restrictions
\begin{eqnarray*}
\gamma_3&=& \gamma_6\\
\gamma_5&\sim& \gamma_1 *\gamma_2\\
\gamma_5* \gamma_3 &\sim& \gamma_4,  
\end{eqnarray*}
which implies that
$$L_3\circ (L_3\times Id)=\{(\gamma_1, \gamma_2, \gamma_3) \mid (\gamma_1*\gamma_2)*\gamma_3 \sim \gamma4\}$$
and since $(\gamma_1*\gamma_2)*\gamma_3 \sim \gamma_1* (\gamma_2* \gamma_3)$ we get that $L_3\circ (L_3\times Id)=L_3\circ (Id\times L_3)$, as we wanted.
\qed
\\

\textbf{\underline{A.5.}} From the definition, we get that
\begin{eqnarray}\label{l1}
L_1&=& \{\gamma \in T^*(PM) \mid \exists \alpha \in T^*PM , \gamma \sim \alpha*\alpha^{-1} \sim \alpha^{-1}*\alpha \}\\
&=&\{ \gamma \in T^*(PM)\mid \gamma \sim (X\equiv x_0, \eta \equiv  0)\}.
\end{eqnarray}
\\

\textbf{\underline{A.6.}}
For the case of $L_2$ it follows from the definition, that
\begin{equation*}
L_2=\{T^*M\mbox{-paths } (\gamma_1, \gamma_2) \in T^*(PM)^2 \mid \gamma_1 \sim \gamma_2 \}.
\end{equation*}
The smoothness for $L_1$ and $L_2$ will be proved in Section \ref{smoothness}.
$\qed$\\

Assuming Theorem \ref{Relational}, it is possible to prove the following

\begin{proposition}
\emph{The relational symplectic groupoid $(G,L,I)$ is regular.}
\end{proposition}
\begin{proof}
It is easy to observe for $C= C_{\partial}$, the space of $T^*M$-paths, that by Proposition \ref{Coi}, $C$ is a Banach submanifold of finite codimension, therefore, axiom \textbf{\underline{A.7.}} holds.
To check \textbf{\underline{A.8.}}, observe that
\begin{equation*}
\underline {L_1}= L_1/ L_2= \{(X,\eta)\in T^*(PM)\mid \exists x_0\in M: (X\equiv x_0, \eta\equiv 0) \}\cong M.
\end{equation*}
We can define the map
\begin{eqnarray*}
s\colon  C &\to& M\\
\gamma=(X,\eta)&\mapsto& X(0) 
\end{eqnarray*}
It follows that $S$, defined in \textbf{\underline{A.9}} corresponds to Graph($s$). The following Lemmata ensure the fact that $dS$ is surjective.
\begin{lemma} \label{mani}\emph{
Let $X$ be a metric space and $PX$ the space of continuous maps from $I$ to $X$. We define the evaluation map
\begin{eqnarray*}
ev_t\colon  PX &\to& X\\
\gamma&\mapsto& \gamma(t).
\end{eqnarray*} 
Then $ev_t$ is a continuous map, provided that $PX$ is equipped with the uniform convergence topology.
}
\end{lemma}
\begin{proof} We fix a path $\gamma \in PX$, a time $t \in T$ and $\varepsilon \in \mathbb R^{>0}$. Consider an open ball $\mathcal U_{\varepsilon}(ev_t(\gamma))$, centered at  $ev_t(\gamma)$ with radius $\varepsilon$. Let $\mathcal V(\gamma):= ev_t^{-1}(U_{\varepsilon}(ev_t(\gamma)))$ and let $\tilde{\gamma} \in \mathcal V(\gamma)$.  The open neighborhood of $\tilde{\gamma}$ defined by 
$$ \mathcal V_{\varepsilon / 2}(\tilde{\gamma}):= \{\xi \in PX \mid d(\tilde{\gamma}, \xi)< \varepsilon/ 2 \}$$
is contained in $V(\gamma)$, therefore 
$$V(\gamma)= \bigcup_{\tilde{\gamma}\in V(\gamma)}V_{\varepsilon / 2}(\tilde{\gamma}), $$
hence, $\mathcal V(\gamma)$ is an open in $PX$, which implies that $ev_t$ is continuous.
\end{proof}
Setting $X=M$, where $M$ is our given smooth manifold, this Lemma proves that the map $s: C\to M$ is continuous, where $C$ is equipped with the subspace topology. This implies that 
Graph$(s)$ is a submanifold of $C\times M$.
To check that it corresponds to a submersion, we will prove the following
\begin{lemma}\emph{ The differential $\delta s$ of the map $s\colon C\to M$ is a well defined surjective map from $TC$ to $TM$}
\end{lemma}
\begin{proof}
Let $\gamma =(X,\eta)\in C$. A vector $\delta \gamma \in T_{\gamma}C$ is described by
$$\delta \gamma = \{ (\delta X, \delta \eta) \mid \delta X \in \Gamma (X^* TM), \, \delta \eta \in \Gamma (X^*T^*M)\}.$$
The map $\delta s$ corresponds to 
\begin{eqnarray*}
\delta s\colon  TC &\to& TM\\
\delta \gamma &\mapsto& \delta X (0),
\end{eqnarray*}
that is the evaluation of $\delta X$ at 0, which is a well defined surjective map, as we wanted.
\end{proof}
\end{proof}
The rest of the Section is devoted to prove the smoothness and the Lagrangian property of the spaces $L_i$ defining the relational symplectic groupoid. 
 
\subsection{Smoothness of $L_i$}\label{smoothness}
In this subsection, we develop the notion of \emph{path holonomy} for the foliated manifold $(T^*PM, \mathcal F)$, where $\mathcal F$ is the characteristic foliation associated to the submanifold $C_{\partial}$, which has codimension $n$, where $n= \dim (M)$. Following the construction in the case of finite dimensional foliations \cite{Mor,Brown}, it is possible to give a smooth manifold structure to the holonomy and monodromy groupoids associated to $(T^*PM, \mathcal F)$. These constructions will allow us to give smoothness conditions to the defining relations $L_i$. First, we recall some basic definitions we will use throughout the proofs.\\

\subsubsection{Foliations for Banach manifolds }
\begin{definition}\label{fol}
\emph{Let $M$ be a connected Banach manifold. Let $$\mathcal F= \{ \mathcal L_{\alpha} \mid \alpha \in A\}$$ be a family of path connected subsets of $M$. Then $(M, \mathcal F)$ is a foliation of codimension $p$ if the following conditions hold:
\begin{enumerate}[1.]
\item $\mathcal L_{\alpha} \cap \mathcal L_{\beta}= \emptyset,$ for $\alpha, \beta \in A, \alpha \neq \beta.$
\item $\bigcup_{\alpha\in A} \mathcal L_{\alpha}=M.$
\item For every $x \in M$, there exists a coordinate chart $(\mathcal U_{\lambda}, \phi_{\lambda})$ for $M$ around $x$ such that for $\alpha \in A$ with $\mathcal U_{\lambda}\cap \mathcal L_{\alpha}\neq \emptyset$, each path connected component of $\phi_{\lambda}(\mathcal U_{\lambda} \cap \mathcal L_{\alpha}) \subset B \times \mathbb R ^p$, where $B$ is a Banach space,  has the form
$$(B \times \{ c \} )\cap \phi (\mathcal U_{\lambda}),$$
where $c \in \mathbb R ^p$ is determined by the path connected component $\mathcal L_{\alpha}$, called a \emph{leaf} of the foliation. If $U$ is a subset of $M$, a path component of the intersection of $U$ with a leaf is called a \emph{plaque} of $U$.
\end{enumerate}}
\end{definition}
Besides the usual finite dimensional examples of foliations, the following proposition gives us characteristic distributions as examples of foliations at the infinite dimensional level.
\begin{proposition} \emph{Let $(M, \omega)$ be a weak symplectic Banach manifold and let $C$ be a coisotropic submanifold such that $TC^ {\perp}$ has finite codimension. Then $TC^{\perp}$ induces a foliation of finite codimension of $C$.}
\end{proposition}
\begin{proof}
We will check first that the distribution $TC^ {\perp}$ is involutive, that is,
$$\omega([X,Y], Z)=0, \forall X,Y \in TC^{\perp}, Z \in TC.$$ 
We know that
\begin{eqnarray*}
d\omega (X,Y,Z)&=& \omega(X, [Y,Z])-\omega(Y, [X,Z])+ \omega (Z, [X,Y])\\
&+& X\omega(Y,Z)-Y\omega(X,Z)+ Z\omega(X,Y)\\
&=&-\omega([X,Y], Z)=0.
\end{eqnarray*}
By the use of Frobenius Theorem for Banach manifolds (for references see \cite{Lang}), this distribution is integrable and it induces a foliation on $C$ of finite codimension.
\end{proof}
In our case of interest the Banach manifold is $\mathcal G= T^*(PM)$ and $C= C_{\partial}$. In \cite{Cat} it is proven that $TC^{\perp}$ has finite codimension. 
Now, we describe the monodromy and holonomy groupoids for foliations.

 \subsubsection{Monodromy groupoid over a foliated manifold}
 
  Let $(M, \mathcal{F})$ be a foliation. The monodromy groupoid, denoted 
by Mon$(M, \mathcal F)$, has as space of objects the manifold $M$ and the space of morphisms is defined as follows:
\begin{itemize}
 \item If $x,y \in M$ belong to the same leaf in the foliation, the morphisms between $x$ and $y$ are homotopy classes, relative to the 
end points, of paths between $x$ and $y$ along the same leaf.
\item If $x$ and $y$ are not in the same leaf, there are no morphims between them.
\end{itemize}
\subsubsection{Holonomy groupoid over a foliated manifold}

We introduce the notion of holonomy for a foliation, that will be useful for our purposes. From now on, $\mathcal L_p$ will denote the leaf on $\mathcal F$ through the point $p$; in this case $p$ should not be confused with the index $\alpha$ in Definition  \ref{fol}, we introduce this new notation for simplicity. 
%\begin{subsubsection{Holonomy of $\mathcal{F}$}}

%\end{subsubsection}

Given $p^{'} \in \mathcal{L}_p$, with $\mathcal{L}_p$ a leaf on $\mathcal{F}$, we consider a path $\alpha_0$ in $\mathcal{L}_p$ such that
$\alpha_0([0,1])\subset U_0$, with $U_0$ given by the foliation chart $(U_0, \phi_0)$. 
Consider $q_0 \in \mathcal{L}_p$ such that $\phi_0(p)$ and $\phi_0(q_0)$ 
lie on the same plaque (i.e in the same leaf with respect to the chart $(U_0, \phi_0)$) and let $T_{p^{'}}$ and $T_{q_0}$ be transversal to 
$\mathcal{F}$ through $p^{'}$ and $q_0$ respectively. A \emph{local holonomy} from $p^{'}$ to $q_0$, denoted by $Hol^{T_{p^{'}},T_{q_0}}(\alpha_0)$
is defined as a germ of a diffeomorphism $f\colon  T_{p^{'}} \to T_{q_0}$, in such a way that there exists an open neighborhood $A$ in $T_{p^{'}}$ where $f$ is a 
leaf preserving diffeomorphism (i.e $a$ and $f(a)$ belong to the same leaf, for $a \in A$). 

Given a foliation and a transversal $T$ through $x \in \mathcal L_{p}$, using the fact that
\[\mbox{Diff}_{x}(T) \cong \mbox{Diff}_0(\mathbb{R}^q) \] 
where $\mbox{Diff}_0$ denotes the group of the germs of diffeomorphisms at 0, $q$ being the codimension of $\mathcal{F}$ and that the holonomy is independent of the homotopy class of the path (up to conjugation with an element in $\mbox{Diff}_0(\mathbb{R}^q$) ), we can see the holonomy as a 
group homomorphism

\[\mbox{hol:} \pi_1(L,x) \to \mbox{Diff}_0(\mathbb{R}^q),\]

The image of this map is denoted by $\mbox{Hol}(L,x)$.\\
Based on this notion, we define the holonomy groupoid of $\mathcal{F}$ in the natural way: the space of objects is the foliated manifold and the 
space of morphisms is the classes of holonomy of paths along the leaves of $\mathcal{F}$. Observe that the isotropy groups of this groupoid are precisely the holonomy 
groups $\mbox{Hol}(L,x)$.\\

\subsubsection{Smoothness of $L_2$}
It can be checked (see \cite{Brown}) that, given a foliated manifold $(M,\mathcal F)$, the equivalence relation $R: M \nrightarrow M$ 
of being in the same leaf, is not necessarily  a smooth submanifold of the cartesian product of the 
foliated manifold with itself. 

Fortunately, there is a way to ``resolve" the singularities, by using the holonomy groupoid associated to what are called \emph{locally Lie groupoids}. Following \cite{Brown, Brown2} we construct the holonomy groupoid associated to the equivalence relation $L_2$, denoted by $\mbox{Hol}(L_2, W)$, where the pair $(L_2,W)$ is the locally Lie groupoid associated to $L_2$ \cite{Brown}.
First, some definitions.
\begin{definition}\emph{
Let $G\rightrightarrows M$ be a groupoid. The \emph{difference} map $\delta: G\times_{(s,s)}G\to G$ is given by $\delta(g,h)=\mu(g, \iota(h))$.}
\end{definition}
\begin{definition}\emph{Let $G\rightrightarrows M$ be a (topological) groupoid. An admissible local section of $G$ is a map $\gamma: U \to G$ from an open set $U$ of $M$ satisfying the following properties:
\begin{enumerate}[1.]
\item $(s\circ \gamma) (x)=x, \forall x\in M$.
\item $(t\circ \gamma)(U)$ is an open in $M$.
\item $(t\circ \gamma): U \to \mathfrak(t\circ \gamma)(U)$ is a homeomorphism.
\end{enumerate}}
\end{definition}
Now, consider  a subspace $M\subset W \subset G$. The triple $(s,t, W)$ is said to have \emph{enough  smooth admissible local sections} \cite{Brown}, if for each $w \in W$ there is an admissible local section $\gamma$ of $G$ satisfying that:
\begin{itemize}
\item $(\gamma \circ s)(w)=w$.
\item $\mathfrak{Im}(\gamma) \subset W$.
\item $\gamma$ is smooth.
\end{itemize}
Now we are able to introduce the notion of locally Lie groupoid:
\begin{definition}\label{locally}\emph{\cite{Brown}. A \emph{locally Lie groupoid} is a pair $(G,W)$, where $G\rightrightarrows M$ is a groupoid and a manifold $W$ such that:
\begin{enumerate}[1.]
\item $M\subset W \subset G$.
\item $W= \iota(W)$.
\item The set $$W_{\delta}:= (W\times_{(s,s)}W)\cap \delta^{-1}(W)$$
is open in $W\times_{(s,s)}W$ and $\delta$ restricted to $W_{\delta}$ is smooth.
\item $s$ and $t$ restricted to $W$ are smooth and $(s,t, W)$ has enough admissible local sections.
\item $W$ generates $G$ as a groupoid.
\end{enumerate}
}
\end{definition}
We will show how $L_2$ can be regarded as a locally Lie groupoid and its associated holonomy groupoid will be the covering manifold which allows us to regard $L_2$ as a morphism in $\mbox{\textbf{Symp}}^{Ext}$.

First, consider the foliated manifold $(M, \mathcal F)$ and a subset $U$ of $M$. We denote $L_2(U)$ the equivalence relation on $U$ defined by 
\[x \sim y \Longleftrightarrow  x \mbox{ and } y \mbox{ are in the same plaque}.\] 
Now, we consider $\Lambda=\{ (\mathcal U_{\lambda}, \phi_{\lambda})\}$ a foliation atlas for $(M, \mathcal F)$ and we define
$$W(\Lambda):= \bigcup _{\mathcal U_{\lambda}} L_2(\mathcal U_{\lambda}),$$
for all domains $\mathcal U_{\lambda}$ of the atlas $\Lambda$.

We prove the following
\begin{proposition} \emph{\cite{Brown}. $W(\Lambda)$, endowed, with the subspace topology with respect to $L_2$ (and hence regarded as a topological subspace of $M \times M$), has the structure of a smooth manifold, coming from the foliated atlas $\Lambda$.}
\end{proposition}
\begin{proof}
The same argument explained in \cite{Brown} works in the case of a foliation on Banach manifold
with finite codimension. There is an induced equivalence relation on $\phi_{\lambda}(\mathcal U_{\lambda}),$ that is determined by the connected components of $\phi_{\lambda}(\mathcal U_{\lambda}) \cap B \times \{ c\} \subset B  \times \mathbb R ^q$
and by using the coordinate function $\phi_{\lambda}$ we induce coordinate charts for $W(\Lambda)$.
\end{proof}
Moreover, it is proven (Theorem 1.3 in \cite{Brown}) that
\begin{theorem}\label{Locally}
Let $(M, \mathcal F)$ be a foliated manifold. Then an atlas $\Lambda$ can be chosen such that $(L_2, W(\Lambda))$ is a locally Lie groupoid. 
\end{theorem}
\begin{remark}\emph{
In \cite{Brown} the construction of the locally Lie groupoid structure on $(L_2, W(\Lambda))$ is done for finite dimensional foliations but it can be naturally extended to the case where the leaf is a Banach manifold and $\mathcal F$ has finite codimension. The only non-trivial step is to check that the property (3) in Definition \ref{locally} is satisfied. For this case, thanks to The Lebesgue Covering Lemma, that can be applied in the Banach case,  there is always a decomposition of a path $a$ from $x$ to $y$ on a leaf $\mathcal L$ in smaller paths $a_i$ such that $a_i$ is a path from $x_i$ to $x_{i+1}$, with $x_0=x, x_{n+1}=y$,  with the property that $(x_i,x_{i+1}) \in W(\Lambda)$.}
\end{remark}

In \cite{Brown2}, the holonomy groupoid for a locally topological groupoid is constructed through a universal property, namely:
\begin{theorem}\label{holo}\emph{(Globalisation Theorem)}\cite{Brown2}. Let $(G,W)$ be a locally topological groupoid. Then there is a topological groupoid $H\rightrightarrows N$, a morphism $\phi: H \to G$ of groupoids, and an embedding $i: W \to H$ of $W$ to an open neighborhood of $N$ satisfying the following:
\begin{enumerate}[1.]
\item $\phi$ is the identity on objects, $\phi\circ i(w)=w, \forall w \in W$, $\phi^{-1}(W)$ is open in $H$ and $\phi \mid_{W}: \phi^{-1}(W) \to W$ is continuous.

\item (Universal property). If $A$ is a topological groupoid and $\xi: A \to G$ is a morphism of groupoids satisfying:
\begin{itemize}
\item $\xi$ is the identity on objects.
\item $\xi\mid_{W}: \xi(W) \to W$ is continuous and $\xi^{-1}(W)$ is an open in $A$ and generates $A$.
\item The triple $(s_A,t_A,A)$ has enough continuous admissible local sections,
\end{itemize}
then there is a unique morphism $\xi^{'}: A \to H$ of topological groupoids such that $\phi \xi^{'}=\xi$ and $\xi^{'}a=i\xi a, \, \forall a \in \xi^{-1}(W).$

%\item 
\end{enumerate}
\end{theorem}
The groupoid $H$ is called the holonomy groupoid of the locally topological groupoid $(G,W)$ and is denoted by $Hol(G,W)$. In the smooth setting, due to Theorem \ref{Locally} , we can prove that
\begin{proposition}\emph{$Hol(L_2, W(\Lambda))$ is a Lie groupoid.}
\end{proposition}

%According to the discussion above, the immersed canonical relation $L_2\colon \mathcal G \nrightarrow \mathcal G$ is given by the triple$(L_2, \mbox{Hol}(L_2, W), \phi)$, where $\phi\colon \mbox{Hol}(L_2, W) \to L_2$ is the projection. 

Thus, the immersed canonical relation associated to the equivalence relation $L_2$ is the triple
$(L_2, Hol(L_2, W(\Lambda)), \phi)$, where $\phi$ is the natural projection from the holonomy groupoid to $L_2$. In fact,  $\phi$ is a covering map over $L_2$ as is explained in \cite{Brown}, with $\phi^{-1}(x,y)= Hol(x,\gamma, y)$, that is, the holonomies of paths $\gamma$ between $x$ and $y$.\\
The next step is to adapt the argument to show that $L_1$ and $L_3$ induce immersed canonical relations. 

\subsubsection{Smoothness of $L_1$}
%First, we observe that, according to Equation \ref{l1}
First of all,  we can see $L_1$ 
as a subspace of the characteristic foliation associated to $\mathcal{C}_{\Pi}$. Namely, we can think the elements of $L_1$ as the 
Lie algebroid morphisms connected to the trivial algebroid morphisms by a path along the distribution. More precisely, if we denote 
by $C \subset \mathcal{C}_{\Pi}$ the submanifold corresponding to the trivial Lie algebroid morphisms ($X$ is constant and $\eta$ is 0), then
\begin{equation}
L_1= \{ \sqcup_{\mathcal{L}\in \mathcal{F}}\mathcal{L} \vert \mathcal{L} \cap C \neq \emptyset  \}. 
\end{equation}
The characteristic foliation can be understood as the space of orbits of a gauge group $H$ acting on $\mathcal{C}_{\Pi}$, 
where $H$ corresponds to the group of local diffeomorphisms generated by the flows of the Hamiltonian vector fields associated to the 
Hamiltonian functions:
\[H_{\beta}(X,\eta)=\int_I \langle dX(u)+ \pi^{\#}(X(u))\eta(u), \beta(X(u),u) \rangle , \]
where $\beta\colon  I \to \Omega^1(M)$ and $\beta(0)=\beta(1)=0.$
This action can be written in local coordinates as follows:
%\begin{eqnarray}
% \delta_{\beta} X^i(u) &=&  -\pi^{ij}(X(u))\beta_j(X(u),u) \nonumber \\
 %\delta_{\beta}\eta_i(u) &=& d_u\beta_i(X(u),u)+\partial_i \pi^{jk}(X(u))\eta_j(u)\beta_k(X(u),u) \nonumber 
  
%\end{eqnarray} 

\begin{eqnarray*}
\delta_{\beta} X^i(u) &=& -\pi^{ij}(X(u))\beta_j(X(u),u)    \\
 \delta_{\beta}\eta_i(u) &=& d_u\beta_i(X(u),u)+\partial_i \pi^{jk}(X(u))\eta_j(u)\beta_k(X(u),u). \nonumber \\
\end{eqnarray*}\label{for}

With this prescription, it is easy to check that the submanifold \[S:=\{(X,\eta)\vert X\equiv X_0,\, \eta \equiv 0  \},\] which is an $n$-dimensional submanifold of $\mathcal C_{\partial}$, where $n= \dim M$,  intersects  the foliation neatly, i.e. 
\[ T_x C \cap T \mathcal L_x = \{ 0\}, \forall x \in C\cap L_x.
\]This holds since after the prescribed gauge transformation,
the points of $C$ are trivially stabilized: the gauge transformation preserves fixed the initial and final points of the path, 
and the fact that the space $\mathcal{C}_\Pi$ is invariant under this gauge transformation implies that there is a unique point 
for each leaf and that the tangent to the orbit (which is given precisely by the gauge trnasformations) and the tangent to $S$ are independent. 
Choosing the transversal $S \subset T$ to the foliation $\mathcal F$, the restriction of the holonomy of $\mathcal F$ to $S$, induces the covering 
\[p\colon  Hol(L_2, W(\Lambda))\mid_{L_1} \to L_1,\]
with fibers the holonomy of paths along the fibers over $S$. Thus, the induced immersed canonical relation for $L_1$ is given by $(L_1, Hol(L_2, W(\Lambda))\mid_{L_1}, p)$.

%\begin{proposition}\label{L2}
%$L_2$ is an immersed Banach submanifold of $T^*PM \times T^*PM$ whenever the characteristic foliation has trivial holonomy. 
%\end{proposition}
%\emph{Proof:} 
%It is left to prove that $L_2$ satisfies the conditions of lemma \ref{L2}. 
%This comes easily from the fact that the characteristic distribution associated 
%to $ \mathcal{C}_{\pi}$ induces a foliation of finite codimension and that the 
%space of leaves of this foliation is precisely the space of orbits of the $T^*M-$ homotopy. 
\subsubsection{Smoothness of $L_3$.}
%\begin{proposition}\label{L3}
Here, we describe $L_3$ in a suitable way so we find a smooth covering for it.
%$L_3$ is an immersed Banach submanifold of $T^*PM\times T^*PM\times T^*PM$ 
%\end{proposition}
%\emph{Proof:} 
The idea of the proof is to use the holonomy groupoid for an equivalence relation, understanding the space $L_3$ in terms of an equivalence homotopy relation. 
First of all, a remark:
\begin{remark}\emph{ The $s$ and $t$ fibers are saturated by the leaves of $\mathcal{F}$ restricted to $\mathcal{C}_{\Pi}$.}
\end{remark}

In other words, given the fact that the characteristic foliation can be understood as the space of orbits of gauge transformations, leaving invariant the initial 
and final points of the paths, the equivalence relation determined by $\mathcal{F}$ is finer than the one determined by $s$ or $t$.\\

In a similar way:
\begin{remark}\emph{
The fibers of the the fibered product of maps:
\[(s\times t)\colon \mathcal{C}_{\Pi} \times \mathcal{C}_{\Pi} \to M \times M \] 
are saturated by the leaves of the product foliation $\mathcal{F} \times \mathcal{F}.$
}\end{remark}
In this way, $\mathcal{F} \times \mathcal{F}$ restricts to a foliation $\mathcal{F}_{(s,t)}$ in 
\[\mathcal{C}_{\Pi}\times_{(s,t)}\mathcal{C}_{\Pi}\subset \mathcal{C}_{\Pi}\times \mathcal{C}_{\Pi}:= (s\times t)^{-1}\Delta\]
This restricted foliation has finite codimension, more precisely
\[\mbox{codim}_{\mathcal{C}_{\Pi}\times_{(s,t)}\mathcal{C}_{\Pi}}\mathcal{F}_{(s,t)}=\mbox{codim}_{\mathcal{C}_{\Pi}\times \mathcal{C}_{\Pi}}\mathcal{F}\times\mathcal{F}-\mbox{codim}_{\mathcal{C}_{\Pi}\times \mathcal{C}_{\Pi}}\mathcal{C}_{\Pi}\times_{(s,t)}\mathcal{C}_{\Pi}=2n. \]
In this way, for a triple $(a,b,c) \in L_3$, the pair $(a,b)$ is an element in $(\mathcal{C}_{\Pi}\times \mathcal{C}_{\Pi}, \mathcal{F}_{(s,t)})$. $c$ can be identified with an 
element in $\mathcal{C}_{\Pi}$ via the smooth map 
%\[\tilde{\beta}\colon  (\mathcal{C}_{\Pi}\times_{(s,t)}\mathcal{C}_{\Pi}) \to (\mathcal{C}_{\Pi}, \mathcal{F}) \]
\begin{eqnarray}
\tilde{\beta}\colon  (\mathcal{C}_{\Pi}\times_{(s,t)}\mathcal{C}_{\Pi}) &\to& (\mathcal{C}_{\Pi}, \mathcal{F})\nonumber \\
(a,b) &\to& a\star b
\end{eqnarray}
where
\[ a\star b(t)= \left\{ \begin{array}{rl}
 a(\beta(2t))&, t\in [0,\frac 1 2] \\
 b(\beta(2t-1))&, t \in [\frac 1 2, 1]
       \end{array} \right.\]
and $\beta$ denotes a bump function $\beta: [0,1] \to [0,1]$.
Therefore, it is possible to characterize the space $L_3$ in the following way:
\[L_3=\{(a,b,c)\in (\mathcal{C}_{\Pi}\times_{(s,t)} \mathcal{C}_{\Pi})\times \mathcal{C}_{\Pi} \vert \mathcal{L}_{(\tilde{\beta}(a,b))}=\mathcal{L}_c\}\]
where $\mathcal{L}$ denotes (as before), the orbits of the $T^*M$-homotopy.
Hence, the induced immersed canonical relation for $L_3$ is $(L_3, \mathcal{C}_{\Pi}\times_{(s,t)} \mathcal{C}_{\Pi}, Hol(L_2, W(\Lambda))).$ %Using Lemma 2, the proof of the Proposition is complete.
%\[L_3:= \{(a,b,c)\in \mathcal{C}_{\Pi}\times \mathcal{C}_{\Pi}\times \mathcal{C}_{\Pi} \vert a(1)=b(0) , (a \sqcup b)\to c \}  \]
%where $a \sqcup b$ corresponds to a $T^*M-$path in such way that 
\subsection{Lagrangian property of $L_i$}
First we prove the following 

\begin{proposition}\emph{
The tangent space $TL_2$ is a Lagrangian subspace of $T(T^*(PM)) \oplus \overline{T(T^*(PM))}.$
}
\end{proposition}

\begin{proof}
First we prove the following 
\begin{lemma}
$TC_{\partial}^{\perp} \oplus TC_{\partial}^{\perp} \subset TL_2$.
\end{lemma}
\begin{proof}
To prove this lemma, we observe first that, according to \cite{Cat}, the leaves of the characteristic foliation of $C_{\partial}$ are precisely the orbits of the gauge equivalence relation given by $L_2$ in $C_{\partial}$. Therefore we get that 
\[TL_2= R^{C} \]
as in Equation \ref{char} and therefore we get that
\[ TC_{\partial}^{\perp} \oplus TC_{\partial}^{\perp} \subset TL_2 \subset TC_{\partial} \oplus TC_{\partial}\]
Observe now that the projection of $TL_2$ with respect to the coisotropic reduction of $C_{\partial}$ is precisely the diagonal of $C_{\partial}$ that is a Lagrangian subspace of $TC_{\partial} \oplus \overline{TC_{\partial}}$.
\end{proof}

Now, the space $TL_2$ satisfies the conditions of Proposition \ref{Coisotropic} and therefore $TL_2$ is Lagrangian, as we wanted.
\end{proof}

%We prove the following proposition in the case when the Poisson manifold is integrable.
%\begin{proposition} 
%The manifolds $L_1,L_2$ and $L_3$ are isotropic submanifolds of $T^{*}PM$, $T^{*}PM\times T^{*}PM$ and $T^{*}PM\times T^{*}PM \times T^{*}PM.$ 

%\end{proposition}
\begin{proposition} \label{L1Lag}\emph{The tangent space $TL_1$ is a Lagrangian subspace of  $T(T^*(PM))$.}
\end{proposition}
\begin{proof}
%For the case of $L_2$, I have the following
First, we prove the following 
\begin{lemma} \emph{
$TL_1$ is an isotropic subspace \footnote{The isotropic condition is a general fact for gauge theories with boundary, see e.g. \cite{AlbertoPavel1}, but we give an explicit proof for the reader's ease. }.}
\end{lemma}
\begin{proof}
The direct computation of the tangent space $TL_1$ %\footnote{$TL_1$ is here regarded as a Zarisky space.} 
 yields 
$$T_{\gamma}L_1= (\delta X(t)+ v, \delta \eta(t)) \mid (\delta X(t), \delta \eta(t) \in TC^{\perp}, v \in T_{\gamma(0)}M).$$
Now, considering two vectors in $T_{\gamma}L_1$ denoted by $(\delta_1 X(t)+ v_1, \delta_1 \eta(t))$  and $(\delta_2 X(t)+ v_2, \delta_2 \eta(t))$ we compute in local coordinates
\begin{eqnarray*}
&&\omega((\delta_1 X^i(t)+ v^i_1, \delta_1 \eta_i(t)), (\delta_2 X^i(t)+ v^i_2, \delta_2 \eta_i(t)))\\ &=&\int_0^1(\delta_1X^i(t)+v^i_1)\delta_2\eta_i(t)-(\delta_2X^i(t)+v^i_2)\delta_1\eta_i(t))dt\\
&=&\int_0^1((\delta_1X^i(t)\delta_2\eta_i(t)-\delta_2X^i(t)\delta_1\eta_i(t))dt+\int_0^1v^i_1\delta_2 \eta_i(t)dt-\int_0^1v^i_2\delta_1 \eta_i(t)dt.
\end{eqnarray*}
The first integral vanishes since $C$ is coisotropic. The second and third integrals vanish since
\begin{equation*}
\int_0^1v^i_1\delta_2 \eta_i(t)dt=\int_0^1v^i_2\delta_1 \eta_i(t)dt= \eta_1(1)-\eta_1(0)= \eta_2(1)-\eta_2(0)=0.
\end{equation*}
\end{proof}
Now, since $TL_1$ is isotropic, after reduction we get that
\begin{eqnarray*}
 &\underline{\omega}&([(\delta_1 X^i(t)+ v^i_1, \delta_1 \eta_i(t)],[\delta_2 X^i(t)+ v^i_1, \delta_2 \eta_i(t)])\\
&=&\omega((\delta_1 X^i(t)+ v^i_1, \delta_1 \eta_i(t)), (\delta_2 X^i(t)+ v^i_2, \delta_2 \eta_i(t)))=0.
\end{eqnarray*}
Therefore $\underline{TL_1}$ is isotropic. Now, since 
\[\underline{T_{\gamma}L_1}= \{ v \in T_{\gamma_{0}}M \} \sim T_{\gamma_0}M,\]
we get that
\[\dim \underline{TL_1}= \dim T_xM= n= 1/2 \dim \underline{TC_{\partial}}.\]
This implies that $\underline{TL_1}$ is Lagrangian and then, by applying Proposition \ref{Coisotropic}, we conclude that $TL_1$ is Lagrangian, as we wanted.
\end{proof}
Now, we prove that
\begin{proposition}\emph{
The space $TL_3$ is a Lagrangian subspace of  $$T(T^*(PM)) \oplus T(T^*(PM)) \oplus \overline{T(T^*(PM))}.$$
}
\end{proposition}
\begin{proof}
In order to prove this Proposition, we first prove the following
\begin{lemma}\emph{ Let $\delta \gamma_1$ and $\delta \gamma_2$ be two vectors in $TC_{\gamma_1}^{\perp}$ and $TC_{\gamma_2}^{\perp}$ that are composable. Then $\delta \gamma_1*\delta \gamma_2 \in TC_{\gamma_1*\gamma_2}^{\perp}. $
}
\end{lemma}
\begin{proof}
This follows immediately from the additive property of $\omega$ with respect to concatenation, namely, if $\delta \gamma$ is a vector in $T_{\gamma_1*\gamma_2}C$, then
\begin{eqnarray*}
\omega(\delta \gamma_1*\delta \gamma_2, \delta \gamma)= \alpha_1 \omega(\delta \gamma_1, \delta_\gamma)+ \alpha_2 \omega(\delta \gamma_2, \delta_\gamma)=0,
\end{eqnarray*}
where $\alpha_i$ are factors due to reparametrizations for $\gamma_i$.
\end{proof}
With this Lemma at hand, we can conclude, from Equation \ref{ELE3} that
\[TC^{\perp} \oplus TC^{\perp} \oplus TC^{\perp} \subset TL_3 \subset TC\oplus TC \oplus TC.\]
Now, after reduction we get that
\begin{eqnarray*}
\underline{\omega}\oplus \underline{\omega}\oplus -\underline{\omega}&\,& ([\delta_1\gamma_1] \oplus [\delta_1\gamma_2]\oplus [\delta_1\gamma_3], [\delta_2\gamma_1] \oplus [\delta_2\gamma_2]\oplus [\delta_2\gamma_3])\\
&=& \underline{\omega}([\delta_1\gamma_1, \delta_1\gamma_2])+\underline{\omega}([\delta_2\gamma_1, \delta_2\gamma_2])-\underline{\omega}([\delta_1\gamma_1\ast\delta_1\gamma_2], [\delta_2\gamma_1\ast\delta_2\gamma_2]),
\end{eqnarray*}
that is zero by the additivity property for $\omega$.
This implies that $\underline{L_3}$ is isotropic. Now, by counting dimensions, we get that
the compatibility condition for $\gamma_1$ and $\gamma_2$ give $3\dim (M)$ independent equations (for the initial, final and coinciding point of $\gamma_1, \gamma_2$ and $\gamma_3$).
Hence,
\[\dim(\underline{L_3})= 6\dim(M)-3\dim (M)= 1/2 \dim (\underline{TC}\oplus \underline{TC}\oplus \underline{TC}).\]
This implies that $\underline{TL_3}$ is Lagrangian. By Proposition \ref{Coisotropic} we conclude that $TL_3$ is Lagrangian, as we wanted.

\end{proof}
\begin{remark}\emph{
In a similar way as in $L_3$ it is possible to define $L_n, \, n\geq 3$ as the space of composable 
$n$-tuples $(\gamma_1, \gamma_2, \cdots \gamma_n)$ in $\mathcal G^n$ and by a similar argument, it can be proven that $L_n$ is Lagrangian, for all $n\geq 3$.}\end{remark}

\section{Equivalences of RRSG}
The next step is to connect the construction of the relational symplectic groupoid for $T^*PM$, which is infinite dimensional, with the s-fiber simply connected symplectic Lie groupoid integrating a Poisson manifold.  The connection is given by the following
\begin{theorem} \label{reduction}Let $(M,\Pi)$ be an integrable Poisson manifold. Let $\mathcal{G}$ be the relational symplectic groupoid associated to $T^*PM$ described above and let $G= \underline{C_{\partial}}$ be the symplectic Lie groupoid associated to the characteristic foliation on $C_{\partial}$. Then $\mathcal{G}$ and $G$ are equivalent as relational groupoids.
\end{theorem}
\begin{proof}
This is a direct consequence of Proposition \ref{pro}, since the previously described relational symplectic groupoid is regular.
\end{proof}
Another fact that results useful with the introduction of relational symplectic groupoids is the comparison of different integrations of Poisson manifolds, i.e. we do not restrict only to the case where the symplectic groupoid is $s$-fiber simply connected. The following Proposition (for more details see \cite{Mor} for the more general case of Lie algebroids) relates different symplectic groupoids integrating a given Poisson manifold $(M, \Pi)$.
\begin{proposition}\emph{
Let $G_{ssc} \rightrightarrows M$ be the s-fiber simply connected symplectic groupoid integrating $(M, \Pi)$ and let $G^{'}\rightrightarrows M$ be another s-fiber connected symplectic groupoid integrating $(M, \Pi)$. Then there exists a discrete group $H$ acting on $G_{ssc}$ such that $G= G_{ssc}/ H$ and the quotient map $p\colon  G_{ssc}\to G$ is the unique groupoid morphism that integrates the identity map 
$id: T^*M \to T^*M$.}
\end{proposition}
With this Proposition in mind, we observe that the projection map $p$, being a local diffeomorphism, is naturally compatible with the symplectic structures of $G_{ssc}$ and $G$; therefore it corresponds to a morphism of symplectic groupoids and hence it corresponds to a morphism of relational symplectic groupoids. Moreover, since locally $p^{-1}$ is also a diffeomorphism the adjoint relation $p^{\dagger}$ is also a morphism. Therefore we have the following
\begin{proposition} \label{Integrations}\emph{Let $G\rightrightarrows M$ and $G^{'}\rightrightarrows M$ be two s-fiber connected symplectic groupoid integrating the same Poisson manifold $(M, \Pi)$. Then $(G,L,I)$ and $(G^{'}, L^{'}, I^{'})$ are equivalent as relational symplectic groupoids.}
\end{proposition}
As a result of this proposition we obtain the following
\begin{corollary}\label{int} \emph{
If $M$ is an integrable Poisson manifold, then the relational symplectic groupoid on $T^*PM$ is equivalent to every s-fiber connected symplectic groupoid integrating $M$.}
\end{corollary}
\begin{remark}
\emph{ If $G\rightrightarrows M$ is a symplectic groupoid, regarded as a relational symplectic groupoid, by Corollary \ref{int} and Lemma \ref{Poi} we recover the well known fact, proven originally by Coste, Dazord and Weinstein (Theorem 1.1 in \cite{Cost}), that there exists a unique Poisson structure $\Pi$ on $M$ such that the source map $s$ is a Poisson map.
}
\end{remark}

It is conjectured that the equivalence of different RSG integrating a given Poisson manifolds holds in general; namely
\begin{conjecture}
Let $M$ be a (possibly non integrable) Poisson manifold. Then any two regular relational symplectic groupoids integrating it are equivalent.
\end{conjecture}

It is also conjectured that there is a relationship between the categoroid of Poisson manifolds with coisotropic relations as morphisms and the the categoroid of \textbf{RRSGpd}. More precisely,
\begin{conjecture}
There is an equivalence of categoroids between $\mbox{\textbf{Poiss}}^{Ext}$, the categoroid of Poisson manifolds as objects and immersed coisotropic submanifolds (coisotropic relations), as morphisms, and the categoroid \textbf{RRSGpd} of regular relational symplectic groupoids.
\end{conjecture}
In this case, the functor $Int:\mbox{\textbf{Poiss}}^{Ext} \to \textbf{RRSGpd}$ would be given by the previously constructed relational 
symplectic groupoid through the PSM, and the adjoint functor $P: \textbf{RRSGpd}\to \mbox{\textbf{Poiss}}^{Ext}$ 
is given by the construction of the Poisson structure of the base of a given RRSG, described in Theorem \ref{Theo1Poisson1}. The construction given in \cite{Coiso} should in principle lead to this functoriality condition.
%\section{RRSG and the Poisson category}
%- A morphisms of Poisson manifolds lifts to a morphism of rrsg

%\section{PSM with branes}
%TO DO (if possible in a resonable time)
%- branes

\end{document}